\documentclass[12pt]{article}

\usepackage{amsmath, amssymb, latexsym, amsthm}
\usepackage{fullpage}
\usepackage{color}
\usepackage{tikz}
\usepackage{graphicx}
\usepackage{stmaryrd,amsbsy,enumerate}
\usepackage{hyperref}
\usetikzlibrary{calc}
\usepackage{cancel}
\usepackage{xspace}
\usepackage{url}
\usepackage{longtable}
\usepackage{rotating}
\usepackage{mathtools,amsmath}
\usepackage{array}
\usepackage{float}
\usepackage{adjustbox}
\usepackage{verbatimbox}
\usepackage{subcaption}

\usepackage{tikz}

\usepackage{soul}

\usepackage[normalem]{ulem}

\newcommand{\rbrac}[1]{\left( #1 \right)}
\newcommand{\sbrac}[1]{\left[ #1 \right]}

\newcommand{\E}{\mathbb{E}}
\newcommand{\ex}{\mathrm{ex}}
\newcommand{\eps}{\epsilon}
\renewcommand{\a}{\alpha}
\renewcommand{\b}{\beta}

\AtBeginDocument{%
   \def\MR#1{}
}

\newtheorem{theorem}{Theorem} 
\newtheorem{theorem*}{Theorem} 

\newtheorem{lemma}[theorem]{Lemma}

\newtheorem*{thm_stability}{Theorem \ref{thm:stability}}

\theoremstyle{definition}

\theoremstyle{remark}
\newtheorem{remark}{Remark}

\newcommand{\vc}[1]{\ensuremath{\vcenter{\hbox{#1}}}}
\tikzset{unlabeled_vertex/.style={inner sep=1.7pt, outer sep=0pt, circle, fill}} 
\tikzset{labeled_vertex/.style={inner sep=2.2pt, outer sep=0pt, rectangle, fill=yellow, draw=black}} 
\tikzset{edge_color0/.style={color=black,line width=1.2pt}} 
\tikzset{edge_color1/.style={color=red,  line width=1.2pt,opacity=0}} 
\tikzset{edge_color2/.style={color=blue, line width=1.2pt,opacity=1}} 
\tikzset{edge_color3/.style={color=green,line width=1.2pt}} 
\tikzset{edge_color4/.style={color=red,  line width=1.2pt,dotted}} 
\tikzset{edge_color5/.style={color=blue, line width=1.2pt,dotted}} 
\tikzset{edge_color6/.style={color=green, line width=1.2pt,dotted}} 
\tikzset{edge_thin/.style={color=black}} 
\tikzset{edge_hidden/.style={color=black,dotted,opacity=0}} 
\tikzset{vertex_color1/.style={inner sep=1.7pt, outer sep=0pt, circle, fill=red, draw=black}} 
\tikzset{vertex_color2/.style={inner sep=1.7pt, outer sep=0pt, circle, fill=blue, draw=black}} 
\tikzset{vertex_color3/.style={inner sep=1.7pt, outer sep=0pt, circle, fill=green, draw=black}} 
\def\outercycle#1#2{ \draw \foreach \x in {0,1,...,#1}{(270-45+\x*360/#2:0.5) coordinate(x\x)};}

\newcommand{\hide}[1]{}
\newcommand{\C}[1]{{\protect\mathcal{#1}}}
\renewcommand{\O}[1]{\overline{#1}}
\newcommand{\tind}{t_{\mathrm{ind}}}
\newcommand{\aut}{\operatorname{aut}}
\newcommand{\dd}{\mathrm{d}}
\newcommand{\I}[1]{{\mathbb #1}}
\newcommand{\V}[1]{{#1}}
\newcommand{\e}{\eps}
\newtheorem{claim}{Claim}[theorem]
\renewcommand{\mid}{:}
\renewcommand{\ge}{\geqslant}
\renewcommand{\geq}{\geqslant}
\renewcommand{\leq}{\leqslant}
\renewcommand{\le}{\leqslant}

\renewcommand{\succeq}{\succcurlyeq}

\begin{document}

\title{Minimizing the number of 5-cycles in graphs with given edge-density}
\author{
Patrick Bennett \thanks{Department of Mathematics, Western Michigan University, Kalamazoo, MI, USA. E-mail: {\tt patrick.bennett@wmich.edu}. Supported in part by Simons Foundation Grant \#426894.}
\and
Andrzej Dudek \thanks{Department of Mathematics, Western Michigan University, Kalamazoo, MI, USA. E-mail: {\tt andrzej.dudek@wmich.edu}. Supported in part by Simons Foundation Grant \#522400.}
\and
Bernard Lidick\'{y} \thanks{Department of Mathematics, Iowa State University. Ames, IA, USA. E-mail: {\tt lidicky@iastate.edu}. Supported in part by NSF grant DMS-1600390.}
\and
Oleg Pikhurko  \thanks{Mathematics Institute and DIMAP,
University of Warwick, Coventry CV4 7AL, UK. E-mail: {\tt o.pikhurko@warwick.ac.uk}. Supported in part by ERC
grant~306493.}
}

\maketitle

\begin{abstract}
Motivated by the work of Razborov about the minimal density of triangles in graphs we study the minimal density of the 5-cycle $C_5$. We show that every graph of order $n$ and size $\left( 1-\frac{1}{k}\right)\binom{n}{2}$, where $k\ge 3$ is an integer, contains at least
\[
\left( \frac{1}{10} -\frac{1}{2k} + \frac{1}{k^2} - \frac{1}{k^3} + \frac{2}{5 k^4} \right)n^5 +o(n^5)
\]
copies of $C_5$. This bound is optimal, since a matching upper bound is given by the balanced complete $k$-partite graph. The proof is based on the flag algebras framework. We also provide a stability result. An SDP solver is not necessary to verify our proofs.
\end{abstract}

\section{Introduction}
It is believed that \emph{extremal graph theory} was started by Tur\'an~\cite{Tur1941} when he proved that any graph on $n$ vertices with more than $\frac{r-2}{2(r-1)}n^2$ edges must contain a copy of $K_r$ (i.e.\ a clique with $r$ vertices). The case $r=3$ was earlier proved by Mantel~\cite{Man1907}.
The general \emph{Tur\'an problem} is to determine the minimum number $\ex(n, H)$ of edges in an $n$ vertex graph that guarantees a copy of a graph $H$, and has been very widely studied. The Erd\H{o}s--Stone Theorem~\cite{ErdSto1946} was a major breakthrough which asymptotically determined the value of $\ex(n, H)$ for all nonbipartite $H$. For such $H$ we have 
\[
\ex(n, H) = \frac{\chi(H)-2}{2(\chi(H)-1)}n^2 + o(n^2).
\]

The natural quantitative question that arises is how many copies of $H$ must be contained in a graph $G$ on $n$ vertices with $m > \ex(n, H)$ edges. This question has also been well studied. Obviously the number of edges $m$ can be expressed as a density parameter $p$ such that $m=p\binom{n}{2}$. Therefore, we will use the following notation. Let $G$ be a (large) graph of order $n$ and  $H$ a small one. Define $\nu_H(G)$ to be the number of unlabeled copies (not necessary induced) of $H$ in~$G$
and the corresponding density as
\[
d_{H}(G) = \frac{\nu_H(G)}{|V(G)|^{|V(H)|}}.
\] 
Furthermore, for a given number $p\in[0,1]$ let
\[
d_H(p) = \lim_{n\to\infty} \min_{G} d_{H}(G),
\] 
where the minimum is taken over all graphs $G$ of order $n$ and size $(p+o(1))\binom{n}{2}$. It is not hard to show by double-counting that the limit exists, see e.g.~\cite[Lemma~2.2]{PikhurkoSliacanTyros19}.

When $H=K_3$ (that means it is a triangle) Moon and Moser~\cite{MooMos1962} and also independently Nordhaus and Stewart \cite{NorSte1963}
determined $d_{K_3}(p)$ for any $p=1-\frac{1}{k}$, where $k$ is a positive integer. We call such $p=1-\frac{1}{k}$ a \emph{Tur\'an density}. Some other partial results for the general $r$-clique $H=K_r$ were established by Lov\'asz and Simonovits~\cite{LovSim1983}. However, for arbitrary $p$ these problems remained open for over 50 years. 

Then Razborov in his seminal paper~\cite{Raz2007} introduced the so-called \emph{flag algebras} and, using them, determined $d_{K_3}(p)$ for any $p$ in~\cite{Raz2008}. Subsequently, Pikhurko and Razborov \cite{PikRaz2017} characterized all almost extremal graphs. Very recently, Liu, Pikhurko and Staden~\cite{LPR2017} found the precise minimum number of triangles among graphs with a given number of edges in almost all range. Nikiforov \cite{Nik2011} determined $d_{K_4}(p)$ for all $p$, and then Reiher~\cite{Rei2016} determined $d_{K_r}(p)$ for all $r$ and~$p$.

In this paper we address the minimum density of the 5-cycle, $C_5$, in a graph with given edge density. We chose to investigate $C_5$ instead of $C_4$ since it is known due to Sidorenko~\cite{Sid1991} that for any fixed constant edge density $p$, the minimum $C_4$-density is achieved asymptotically by the random graph $G_{n, p}$. It is worth mentioning some other research related to 5-cycles.  Specifically, Grzesik~\cite{Grz2012} and independently Hatami, Hladk\'y, Kr{\'al'}, Norine and Razborov~\cite{HatHlaKraNorRaz2013} proved that the maximum density of 5-cycles in a triangle-free graph that is large or its number of vertices is a power of 5 is achieved by the balanced blow-up of a 5-cycle. 
The extension to graphs of all sizes, with one exception on 8 vertices, was done by Lidick\'y and Pfender~\cite{LP2017}.
This settled in the affirmative a conjecture of Erd\H{o}s~\cite{Erd1984}. On the other hand, Balogh, Hu, Lidick\'y, and Pfender~\cite{BalHuLidPfe2016} studied the problem of maximizing induced 5-cycles, and proved that this is achieved by the balanced iterated blow-up of a 5-cycle. This confirmed a special case of a conjecture of Pippinger and Golumbic~\cite{PipGol1975}.

The main result of this paper is as follows.

\begin{theorem}\label{thm:main}
Let $k\ge 3$ be an integer. Define 
 \begin{equation}\label{eq:p}
 p=1-\frac{1}{k}\quad\mbox{and}\quad\lambda=\frac{1}{10} -\frac{1}{2k} + \frac{1}{k^2} - \frac{1}{k^3} + \frac{2}{5 k^4}.
 \end{equation}
 Then 
 $$d_{C_5}(p)=\lambda.$$
\end{theorem}
\noindent

We also have the following stability result. Let the \emph{Tur\'an graph} $T_k^n$ be the complete $k$-partite graph on $n$ vertices with part sizes as equal as possible.

\begin{theorem}\label{th:stab} For every integer $k\ge 3$ and real $\delta>0$ there is $\eps>0$ such that every graph $G$ with $n\ge 1/\eps$ vertices, at least $(p-\eps){n\choose 2}$ edges  and at most $(\lambda +\eps)n^5$
	copies of $C_5$ is within edit distance $\delta n^2$ from the Tur\'an graph $T_k^n$, where $p$ and $\lambda$ are as in~\eqref{eq:p}. 
\end{theorem}

Observe that the above theorems (as stated) also hold in the case $k=2$ for which $d_{C_5}(\frac{1}{2})=0$. However, their validity in this case easily follows from known standard results. 
Although the proofs of Theorems~\ref{thm:main} and~\ref{th:stab} are based on the flag algebras framework, their verification does not require using any SDP solver.

Theorems~\ref{thm:main} and~\ref{th:stab} are proved in respectively Sections~\ref{sec:main_thm} and~\ref{sec:stability}.
Finally, in Section~\ref{sec:general_case}, we discuss the general edge density and provide an upper bound on $d_{C_5}(p)$ for any $p\in [0,1]$.

\section{Proof of the main theorem}\label{sec:main_thm}

\subsection{Upper bound}

By considering the sequence of graphs $T_k^n$ as $n\to\infty$, we get
\[
d_{C_5}(T_k^n) = \frac{\left[\frac{1}{10} (k)_5 + \frac12 (k)_4 + \frac 12 (k)_3 \right] \left(\frac nk\right)^5}{n^5} + o(1),
\]
where $(k)_{\ell} = k(k-1)\cdots (k-\ell+1)$ is the \emph{falling factorial}.
To justify the numerator, we count the number of $C_5$ copies with vertices in parts $V_1, V_2, V_3, V_4, V_5$ of the partition. These parts may not all be distinct: for example we may have $V_1 = V_3$. However $T_k^n$ has no edges within these parts and so we know $V_i \neq V_{i+1}$. We count copies of $C_5$ by grouping them according to how many distinct parts there are among $V_1, \ldots, V_5$. Now there are asymptotically $\frac{1}{10} (k)_5 \left(\frac nk\right)^5$ copies that hit 5 different parts (label 5 distinct parts, choose one vertex in each part, and divide by 10 for overcounting). Also, there are  asymptotically $\frac12 (k)_4\left(\frac nk\right)^5$ copies hitting 4 parts, and $\frac12 (k)_3\left(\frac nk\right)^5$ copies hitting 3 parts. 

Simplifying, we get that $d_{C_5}(T_k^n) =\lambda+o(1)$,
which implies the upper bound in Theorem~\ref{thm:main}.

\subsection{Lower bound}\label{sec:lower}

\subsubsection{Preliminaries}

The proof of the lower bound in Theorem~\ref{thm:main} relies on the celebrated flag algebra method introduced by Razborov~\cite{Raz2007}. Here we briefly discuss the main idea behind this approach, referring the reader to~\cite{Raz2007} for all details. Alternatively, our lower bound is rephrased at the beginning of Section~\ref{sec:stability} by means of a combinatorial identity (namely~\eqref{eq:main}) whose statement does not use any flag algebra formalism.

Let $(G_n)_{n\in\mathbb{N}}$ be a sequence of graphs, such that order of $G_n$ increases.
Such a sequence is called \emph{convergent} if for every fixed graph $H$, the density of $H$ in $G_n$ converges, i.e., for every  $H$ there exists some number $\phi(H)$, such that
\[
\lim_{n \to \infty} p(H,G_n) = \phi(H),
\]
where $p(H,G)$ is the probability that $|H|=|V(H)|$ vertices chosen uniformly at random from $V(G)$ induce a copy of $H$. (Here, it will be more convenient to count induced copies of $H$; see e.g.\ Equations~(5.19)--(5.21) in~\cite{Lovasz:lngl}
that show how to switch between induced and non-induced versions.)
Notice that any sequence of graphs whose orders increase has a convergent subsequence. 
Thus, without loss of generality we assume $G_n$ is convergent.
Note that $\phi$ cannot be an arbitrary function since it must satisfy many obvious identities such as $\phi(edge) + \phi(nonedge)=1$.

Interestingly, these $\phi$ exactly correspond to homomorphisms that we now describe.
Denote by $\mathcal{F}$ the set of all graphs and by $\mathcal{F}_\ell$ the set of graphs of order $\ell$, up to an isomorphism.
Let $\mathbb{R}\mathcal{F}$ be the set of all finite formal linear combinations of graphs in $\mathcal{F}$ with real coefficients. It comes with the natural operations of addition and multiplication by a real number.
Let $\mathcal{K}$ be a linear subspace generated by all linear combinations
\begin{align}\label{eq:zero}
F - \sum_{H \in \mathcal{F}_\ell} p(F,H) \cdot H, 
\end{align}
where $\ell > |F|$. Notice that $\phi$ evaluated at any element of $\mathcal{K}$ gives 0.
Finally, let $\mathcal{A}$ be $\mathbb{R}\mathcal{F}$ factorized by $\mathcal{K}$.
It is possible to define multiplication on $\mathcal{A}$, which we do in Section~\ref{subsection:cFM}.
It can be proved that $\mathcal{A}$ is indeed an algebra. 
Now limits of convergent graph sequences correspond to homomorphism $\phi$ from $\mathcal{A}$ to $\mathbb{R}$ such that $\phi(F) \geq 0$ for all $F \in \mathcal{F}$.
Denote the set of all such homomorphisms by $Hom^+(\mathcal{A},\mathbb{R})$.

Let $OPT$ be the following linear combination, which counts the $C_5$ copies using induced subgraphs:
\[
OPT = 
\vc{\begin{tikzpicture}\outercycle{6}{5}
\draw[edge_color2] (x0)--(x1);\draw[edge_color1] (x0)--(x2);\draw[edge_color1] (x0)--(x3);\draw[edge_color2] (x0)--(x4);  \draw[edge_color2] (x1)--(x2);\draw[edge_color1] (x1)--(x3);\draw[edge_color1] (x1)--(x4);  \draw[edge_color2] (x2)--(x3);\draw[edge_color1] (x2)--(x4);  \draw[edge_color2] (x3)--(x4);    
\draw (x0) node[unlabeled_vertex]{};\draw (x1) node[unlabeled_vertex]{};\draw (x2) node[unlabeled_vertex]{};\draw (x3) node[unlabeled_vertex]{};\draw (x4) node[unlabeled_vertex]{};
\end{tikzpicture}}
+ \vc{\begin{tikzpicture}\outercycle{6}{5}
\draw[edge_color2] (x0)--(x1);\draw[edge_color1] (x0)--(x2);\draw[edge_color1] (x0)--(x3);\draw[edge_color2] (x0)--(x4);  \draw[edge_color2] (x1)--(x2);\draw[edge_color1] (x1)--(x3);\draw[edge_color1] (x1)--(x4);  \draw[edge_color2] (x2)--(x3);\draw[edge_color2] (x2)--(x4);  \draw[edge_color2] (x3)--(x4);    
\draw (x0) node[unlabeled_vertex]{};\draw (x1) node[unlabeled_vertex]{};\draw (x2) node[unlabeled_vertex]{};\draw (x3) node[unlabeled_vertex]{};\draw (x4) node[unlabeled_vertex]{};
\end{tikzpicture}}
+ \vc{\begin{tikzpicture}\outercycle{6}{5}
\draw[edge_color2] (x0)--(x1);\draw[edge_color1] (x0)--(x2);\draw[edge_color1] (x0)--(x3);\draw[edge_color2] (x0)--(x4);  \draw[edge_color2] (x1)--(x2);\draw[edge_color1] (x1)--(x3);\draw[edge_color2] (x1)--(x4);  \draw[edge_color2] (x2)--(x3);\draw[edge_color2] (x2)--(x4);  \draw[edge_color2] (x3)--(x4);    
\draw (x0) node[unlabeled_vertex]{};\draw (x1) node[unlabeled_vertex]{};\draw (x2) node[unlabeled_vertex]{};\draw (x3) node[unlabeled_vertex]{};\draw (x4) node[unlabeled_vertex]{};
\end{tikzpicture}}
+ 2 \vc{\begin{tikzpicture}\outercycle{6}{5}
\draw[edge_color2] (x0)--(x1);\draw[edge_color1] (x0)--(x2);\draw[edge_color1] (x0)--(x3);\draw[edge_color2] (x0)--(x4);  \draw[edge_color2] (x1)--(x2);\draw[edge_color2] (x1)--(x3);\draw[edge_color1] (x1)--(x4);  \draw[edge_color2] (x2)--(x3);\draw[edge_color2] (x2)--(x4);  \draw[edge_color2] (x3)--(x4);    
\draw (x0) node[unlabeled_vertex]{};\draw (x1) node[unlabeled_vertex]{};\draw (x2) node[unlabeled_vertex]{};\draw (x3) node[unlabeled_vertex]{};\draw (x4) node[unlabeled_vertex]{};
\end{tikzpicture}} 
+ 2 \vc{\begin{tikzpicture}\outercycle{6}{5}
\draw[edge_color2] (x0)--(x1);\draw[edge_color1] (x0)--(x2);\draw[edge_color1] (x0)--(x3);\draw[edge_color2] (x0)--(x4);  \draw[edge_color2] (x1)--(x2);\draw[edge_color2] (x1)--(x3);\draw[edge_color2] (x1)--(x4);  \draw[edge_color2] (x2)--(x3);\draw[edge_color2] (x2)--(x4);  \draw[edge_color2] (x3)--(x4);    
\draw (x0) node[unlabeled_vertex]{};\draw (x1) node[unlabeled_vertex]{};\draw (x2) node[unlabeled_vertex]{};\draw (x3) node[unlabeled_vertex]{};\draw (x4) node[unlabeled_vertex]{};
\end{tikzpicture}}
+ 4 \vc{\begin{tikzpicture}\outercycle{6}{5}
\draw[edge_color2] (x0)--(x1);\draw[edge_color1] (x0)--(x2);\draw[edge_color2] (x0)--(x3);\draw[edge_color2] (x0)--(x4);  \draw[edge_color2] (x1)--(x2);\draw[edge_color1] (x1)--(x3);\draw[edge_color2] (x1)--(x4);  \draw[edge_color2] (x2)--(x3);\draw[edge_color2] (x2)--(x4);  \draw[edge_color2] (x3)--(x4);    
\draw (x0) node[unlabeled_vertex]{};\draw (x1) node[unlabeled_vertex]{};\draw (x2) node[unlabeled_vertex]{};\draw (x3) node[unlabeled_vertex]{};\draw (x4) node[unlabeled_vertex]{};
\end{tikzpicture}} 
+ 6 \vc{\begin{tikzpicture}\outercycle{6}{5}
\draw[edge_color2] (x0)--(x1);\draw[edge_color1] (x0)--(x2);\draw[edge_color2] (x0)--(x3);\draw[edge_color2] (x0)--(x4);  \draw[edge_color2] (x1)--(x2);\draw[edge_color2] (x1)--(x3);\draw[edge_color2] (x1)--(x4);  \draw[edge_color2] (x2)--(x3);\draw[edge_color2] (x2)--(x4);  \draw[edge_color2] (x3)--(x4);    
\draw (x0) node[unlabeled_vertex]{};\draw (x1) node[unlabeled_vertex]{};\draw (x2) node[unlabeled_vertex]{};\draw (x3) node[unlabeled_vertex]{};\draw (x4) node[unlabeled_vertex]{};
\end{tikzpicture}} 
+ 12 \vc{\begin{tikzpicture}\outercycle{6}{5}
\draw[edge_color2] (x0)--(x1);\draw[edge_color2] (x0)--(x2);\draw[edge_color2] (x0)--(x3);\draw[edge_color2] (x0)--(x4);  \draw[edge_color2] (x1)--(x2);\draw[edge_color2] (x1)--(x3);\draw[edge_color2] (x1)--(x4);  \draw[edge_color2] (x2)--(x3);\draw[edge_color2] (x2)--(x4);  \draw[edge_color2] (x3)--(x4);    
\draw (x0) node[unlabeled_vertex]{};\draw (x1) node[unlabeled_vertex]{};\draw (x2) node[unlabeled_vertex]{};\draw (x3) node[unlabeled_vertex]{};\draw (x4) node[unlabeled_vertex]{};
\end{tikzpicture}},
\]
where the coefficient of each graph is the number of copies of $C_5$ it contains. Thus, 
\begin{equation}\label{eq:OPTAsC5}
\phi(OPT) = 120 \lim_{n \rightarrow \infty} d_{C_5}(G_n).
\end{equation}
The factor $120=5!$ comes from the fact that $p(F,G_n)$ for $F\in\C F_5$ is the number of copies of $F$ divided by ${n\choose 5}$ whereas our scaling for $d_{C_5}$ was chosen as~$n^{-5}$.
Notice that $OPT$ can be written as a linear combination of all 34 graphs on 5-vertices, where 26 graphs have coefficient 0. Namely,
\begin{align}\label{eq:FA1}
OPT = \sum_{F \in \mathcal{F}_5} c_F^{OPT} F,
\end{align}
where the nonzero entries $c_F^{OPT}$ are as above.

Our goal is to prove a good lower bound on
\[
\min_{\phi \in Hom^+(\mathcal{A},\mathbb{R})} \phi(OPT),
\]
 given that that the edge density is $p$, that is, we have
\begin{align}\label{eq:FA2}
\phi\left(
 \vc{\begin{tikzpicture}\outercycle{3}{2}
\draw[edge_color2] (x0)--(x1);
\draw (x0) node[unlabeled_vertex]{};\draw (x1) node[unlabeled_vertex]{};
\end{tikzpicture}}\right) = p.
\end{align}

 For this we find suitable $A \in \mathcal{A}$, such that  $\phi(A) \geq 0$ for all $\phi \in Hom^+(\mathcal{A},\mathbb{R})$ with $\phi(K_2)=p$, and use it in calculations. In particular, we will use it as $\phi(OPT) \geq \phi(OPT) - \phi(A) = \phi(OPT - A) \geq c$, where $c$ is the smallest coefficient $c_F$ when we express   $OPT-A$ as
 $\sum_{F\in\C F_\ell} c_FF$. Note that $A$ may contain both positive and negative coefficients, and these coefficients combine with coefficients in $OPT$.

When $p=1-\frac1k$ for integer $k\ge 3$, it is possible to prove the sharp lower bound as above by considering graphs of order 5 with only one labeled vertex. Similarly to defining the algebra $\mathcal{A}$ and limits of convergent graph sequences, one can define limits of sequences from the set $\mathcal F^1$ which consists of graphs with exactly one labeled vertex up a label-preserving isomorphism. 
This gives an algebra $\mathcal{A}^1$  and homomorphisms $Hom^+(\mathcal{A}^1,\mathbb{R})$.
In the following, we depict the labeled vertex by a square.

Let $X$ be the following column vector
\begin{equation}\label{eq:X}
X = (X_1,\dots,X_6)^T=\left( 
\vc{\begin{tikzpicture}\outercycle{4}{3}
\draw[edge_color1] (x0)--(x1);\draw[edge_color1] (x0)--(x2);  \draw[edge_color1] (x1)--(x2);    
\draw (x0) node[labeled_vertex]{};\draw (x1) node[unlabeled_vertex]{};\draw (x2) node[unlabeled_vertex]{};
\end{tikzpicture}}, 
\vc{\begin{tikzpicture}\outercycle{4}{3}
\draw[edge_color1] (x0)--(x1);\draw[edge_color1] (x0)--(x2);  \draw[edge_color2] (x1)--(x2);    
\draw (x0) node[labeled_vertex]{};\draw (x1) node[unlabeled_vertex]{};\draw (x2) node[unlabeled_vertex]{};
\end{tikzpicture}},
\vc{\begin{tikzpicture}\outercycle{4}{3}
\draw[edge_color1] (x0)--(x1);\draw[edge_color2] (x0)--(x2);  \draw[edge_color1] (x1)--(x2);    
\draw (x0) node[labeled_vertex]{};\draw (x1) node[unlabeled_vertex]{};\draw (x2) node[unlabeled_vertex]{};
\end{tikzpicture}},
\vc{\begin{tikzpicture}\outercycle{4}{3}
\draw[edge_color1] (x0)--(x1);\draw[edge_color2] (x0)--(x2);  \draw[edge_color2] (x1)--(x2);    
\draw (x0) node[labeled_vertex]{};\draw (x1) node[unlabeled_vertex]{};\draw (x2) node[unlabeled_vertex]{};
\end{tikzpicture}},
\vc{\begin{tikzpicture}\outercycle{4}{3}
\draw[edge_color2] (x0)--(x1);\draw[edge_color2] (x0)--(x2);  \draw[edge_color1] (x1)--(x2);    
\draw (x0) node[labeled_vertex]{};\draw (x1) node[unlabeled_vertex]{};\draw (x2) node[unlabeled_vertex]{};
\end{tikzpicture}},
\vc{\begin{tikzpicture}\outercycle{4}{3}
\draw[edge_color2] (x0)--(x1);\draw[edge_color2] (x0)--(x2);  \draw[edge_color2] (x1)--(x2);    
\draw (x0) node[labeled_vertex]{};\draw (x1) node[unlabeled_vertex]{};\draw (x2) node[unlabeled_vertex]{};
\end{tikzpicture}}
\right)^T.
\end{equation}
Notice that $X$ is the vector of all graphs on 3 vertices with exactly one labeled vertex (the yellow square).
For isomorphism, the labeled vertex must be preserved but the remaining two
vertices may be swapped.
If  $M$ is a positive semidefinite matrix in $\mathbb{R}^{6 \times 6}$, then for
every $\phi^1 \in Hom^+(\mathcal{A}^1,\mathbb{R})$ it holds that
\[
\phi^1\left( X^T M X \right)=\phi^1\left(X^T\right)  M \phi^1\left(X\right)\ge 0,
\]
where by $\phi^1\left(X\right)$ we mean the vector that results from applying $\phi^1$ to each coordinate of $X$.

Also, there is a linear operator  $\left\llbracket\  \cdot\ \right\rrbracket_1:\mathbb{R}\mathcal F^1\to \mathbb{R}\mathcal F$ (which, roughly speaking, ``unlabels'' each $F\in\mathcal F^1$) such that for all $\phi \in Hom^+(\mathcal{A},\mathbb{R})$ we have $\phi\left(  \llbracket X^T M X \rrbracket_1  \right)\ge 0$. Furthermore,  we have
\begin{align}\label{eq:FA3}
\llbracket X^T M X \rrbracket_1 &= \sum_{F \in \mathcal{F}_5} c_F^M \cdot F,
\end{align}
 see Section~\ref{subsection:cFM} for more details, in particular on how to calculate coefficients $c_F^M$.

Also, the relation \eqref{eq:zero} for cliques $K_2$ and $K_1$ gives that respectively
$K_2  = \sum_{H \in \mathcal{F}_5} p(K_2,H) \cdot H$ and
$1=K_1=\sum_{H \in \mathcal{F}_5} H$. Thus
\eqref{eq:FA2} can be written as an identity involving densities of $5$-vertex graphs.

 Next, we take the sum of equations \eqref{eq:FA1}, \eqref{eq:FA2} multiplied by some $\alpha$, and $\phi\left(  \llbracket X^T M X \rrbracket_1  \right)\ge 0$ expanded using \eqref{eq:FA3}, and obtain 
\begin{align*}
\phi(OPT) &\geq
\phi(OPT) + 
\alpha\left(
p-\phi\left(
 \vc{\begin{tikzpicture}\outercycle{3}{2}
\draw[edge_color2] (x0)--(x1);
\draw (x0) node[unlabeled_vertex]{};\draw (x1) node[unlabeled_vertex]{};
\end{tikzpicture}}\right)
\right)
-
\phi\left(  \llbracket X^T M X \rrbracket_1  \right)\\
&=  \phi\left(  OPT  +\alpha p -\alpha  \vc{\begin{tikzpicture}\outercycle{3}{2}
\draw[edge_color2] (x0)--(x1);
\draw (x0) node[unlabeled_vertex]{};\draw (x1) node[unlabeled_vertex]{};
\end{tikzpicture}}
-
\llbracket X^T M X \rrbracket_1
\right)\\
&= \phi\left(  \sum_{F \in \mathcal{F}_5} \left(c_F^{OPT} + \alpha p -\alpha \cdot p(K_2,F) - c_F^M\right) \cdot  F   \right)
\end{align*}
(In Appendix~\ref{appendix:A} we provide $c_F^{OPT}$ and $p(K_2,F)$ for each $F\in \mathcal{F}_5$.) For $F\in\mathcal F_5$, define
\begin{equation}\label{eq:cF}
c_F =  c_F^{OPT} + \alpha p -\alpha \cdot p(K_2,F) - c_F^M.
\end{equation}
With this notation
\begin{align}
\phi(OPT) \geq \phi\left( \sum_{F \in \mathcal{F}_5}  c_F \cdot F      \right) 
\geq   \min_{F \in \mathcal{F}_5}  c_F \cdot \phi\left(   \sum_{F \in \mathcal{F}_5}  F  \right) 
 = \min_{F \in \mathcal{F}_5}  c_F,\label{eq:FA}
\end{align}
where $c_F$ is a number that depends on the choice of  $M$ and $\alpha$. Let us transfer this back to our extremal graph problem:

\begin{lemma}\label{lm:lower} For every $p\in[0,1]$, $M\succeq 0$ and $\alpha\in\I R$, with $c_F=c_F(p,M,\alpha)$ as in~\eqref{eq:cF}, we have 
	$$
	d_{C_5}(p)\ge \frac1{120} \min_{F \in \mathcal{F}_5}  c_F.
	$$
	\end{lemma}
\begin{proof} Suppose on the contrary we can find an increasing sequence of graphs $G_n$ with edge density $p+o(1)$ such that $d_5(G_n)$  stays strictly below the stated bound. Take a convergent subsequence and let $\phi\in Hom^*(\C A,\I R)$ be its limit. It satisfies~\eqref{eq:FA2} so the bound in~\eqref{eq:FA} applies to~$\phi$. However, this contradicts~\eqref{eq:OPTAsC5}.\end{proof}

\subsubsection{Finding the optimum}\label{subsec:nice_p}

Let an integer $k \geq 3$ be fixed. Let $p$ and $\lambda$ be as in~\eqref{eq:p}.
By Lemma~\ref{lm:lower}, in order to finish the proof of Theorem~\ref{thm:main}, it is enough to present some $M\succeq 0$  and $\alpha\in\I R$ with $c_F \ge 5!\,\lambda$ for every $F \in \mathcal{F}_5$.
Let 
\[
 \alpha=\frac{1}{k^3}\left(60k^3 - 240k^2 + 360k - 192\right).
\]
In order to define the matrix $M$ we define first two matrices $A$ and $B$ as follows:
\[
A = 
.
\]
It is easy to verify (by checking principal minors) that $A$ is positive definite for any $k \geq 3$. Therefore, the matrix
\begin{equation}\label{eq:BAB}
M = \frac{3}{2k^4} B^T A B
\end{equation}
is positive semidefinite. In Section~\ref{subsec:guessing} we briefly describe how we determined matrices $A$ and~$B$.
With this choice of $M$ and $\alpha$, one can verify using for example Maple (see Appendix~\ref{appendix:B}) that coefficients $c_F$ satisfy:
\begin{align*}
c_{\scalebox{0.3}{%
}}
&= \frac{1}{5k^4}(61k^4-300k^3+600k^2-600k+240).
\end{align*}
Since the entries only ever disagree in the $k^4$ coefficient, it is easy to see that the smallest $c_F$'s are in the first two rows and are equal to $5!\,\lambda$, as desired. (Recall that this proves the lower bound on $d_{C_5}(p)$ of Theorem~\ref{thm:main} by Lemma~\ref{lm:lower}.)
\hide{
	. Thus we get the desired lower bound:
$$
\phi(OPT) \ge \min_{F \in \mathcal{F}_5}  c_F 
=  \frac{1}{5k^4}(60k^4-300k^3+600k^2-600k+240) = 5!\,\lambda.
$$
}

\subsubsection{Products of graphs and determining $c_F^M$ coefficients}\label{subsection:cFM}

First, we define the product of unlabeled graphs. Recall that for a graph $G$ we denote $|V(G)|$  by $|G|$.
Let $F_1,F_2,F$ in $\mathcal{F}$ be such that $|F_1| + |F_2| \leq  |F|$.
Choose uniformly at random two disjoint subsets $X_1$ and $X_2$ of $V(F)$ of sizes $|F_1|$ and $|F_2|$, respectively.
Denote by $p(F_1,F_2;F)$ the probability that $F[X_1]$ is isomorphic to $F_1$ and $F[X_2]$ is isomorphic to $F_2$.
Finally, the product of $F_1$ and $F_2$ is defined as
\[
	F_1 \times F_2 = \sum_{F \in \mathcal{F}_{|F_1|+|F_2|}} p(F_1,F_2;F)\cdot F.
\]
The product can be extended to linear combinations of graphs and gives a multiplication operation in $\mathcal{A}$. 

The product in $\mathcal{A}^1$ is defined along the same lines as in $\mathcal{A}$ but the intersection of $X_1$ and $X_2$ is exactly the labeled vertex. A more precise definition follows.
Let $F_1,F_2,F$ in $\mathcal{F}^1$ such that $|F_1| + |F_2| \leq  |F|-1$.
Choose uniformly at random subsets $X_1$ and $X_2$ of $V(F)$ of sizes $|F_1|$ and $|F_2|$, respectively whose intersection is exactly the one labeled vertex.
Denote by $p(F_1,F_2;F)$ the probability that $F[X_1]$ is isomorphic to $F_1$ and $F[X_2]$ is isomorphic to $F_2$, where isomorphism preserves the labeled vertex.
Finally, the product of $F_1$ and $F_2$ is defined as
\[
	F_1 \times F_2 = \sum_{F \in \mathcal{F}_{|F_1|+|F_2|-1}} p(F_1,F_2;F)\cdot F.
\]

Next we define the unlabeling operator $\left\llbracket\  \cdot\ \right\rrbracket_1: \mathcal{F}^1 \to \mathbb{R}\mathcal{F}$. We extend  $\left\llbracket\  \cdot\ \right\rrbracket_1$ to a linear function $\mathbb{R}\mathcal{F}^1 \to \mathbb{R}\mathcal{F}$ which we also call $\left\llbracket\  \cdot\ \right\rrbracket_1$.
Let $F \in \mathcal{F}^1$. Denote by $G \in \mathcal{F}$ the graph obtained from $F$ by unlabeling the labeled vertex.  
Let $v$ be a vertex in $G$ chosen uniformly at random. 
Let $q$ be the probability that $G$ with labeled $v$ is isomorphic to $F$. Then 
\[
\left\llbracket F \right\rrbracket_1 = q \cdot G.
\]

Recall that $X$ is the vector of all 3-vertex labeled graphs from $\C F^1$: 
\[
X = (X_1,X_2,X_3,X_4,X_5,X_6)^T = \left( 
\vc{\begin{tikzpicture}\outercycle{4}{3}
\draw[edge_color1] (x0)--(x1);\draw[edge_color1] (x0)--(x2);  \draw[edge_color1] (x1)--(x2);    
\draw (x0) node[labeled_vertex]{};\draw (x1) node[unlabeled_vertex]{};\draw (x2) node[unlabeled_vertex]{};
\end{tikzpicture}}, 
\vc{\begin{tikzpicture}\outercycle{4}{3}
\draw[edge_color1] (x0)--(x1);\draw[edge_color1] (x0)--(x2);  \draw[edge_color2] (x1)--(x2);    
\draw (x0) node[labeled_vertex]{};\draw (x1) node[unlabeled_vertex]{};\draw (x2) node[unlabeled_vertex]{};
\end{tikzpicture}},
\vc{\begin{tikzpicture}\outercycle{4}{3}
\draw[edge_color1] (x0)--(x1);\draw[edge_color2] (x0)--(x2);  \draw[edge_color1] (x1)--(x2);    
\draw (x0) node[labeled_vertex]{};\draw (x1) node[unlabeled_vertex]{};\draw (x2) node[unlabeled_vertex]{};
\end{tikzpicture}},
\vc{\begin{tikzpicture}\outercycle{4}{3}
\draw[edge_color1] (x0)--(x1);\draw[edge_color2] (x0)--(x2);  \draw[edge_color2] (x1)--(x2);    
\draw (x0) node[labeled_vertex]{};\draw (x1) node[unlabeled_vertex]{};\draw (x2) node[unlabeled_vertex]{};
\end{tikzpicture}},
\vc{\begin{tikzpicture}\outercycle{4}{3}
\draw[edge_color2] (x0)--(x1);\draw[edge_color2] (x0)--(x2);  \draw[edge_color1] (x1)--(x2);    
\draw (x0) node[labeled_vertex]{};\draw (x1) node[unlabeled_vertex]{};\draw (x2) node[unlabeled_vertex]{};
\end{tikzpicture}},
\vc{\begin{tikzpicture}\outercycle{4}{3}
\draw[edge_color2] (x0)--(x1);\draw[edge_color2] (x0)--(x2);  \draw[edge_color2] (x1)--(x2);    
\draw (x0) node[labeled_vertex]{};\draw (x1) node[unlabeled_vertex]{};\draw (x2) node[unlabeled_vertex]{};
\end{tikzpicture}}
\right)^T.
\]

\newcommand{\Xa}{ \vc{\begin{tikzpicture}\outercycle{4}{3}
\draw[edge_color1] (x0)--(x1);\draw[edge_color1] (x0)--(x2);  \draw[edge_color1] (x1)--(x2);    
\draw (x0) node[labeled_vertex]{};\draw (x1) node[unlabeled_vertex]{};\draw (x2) node[unlabeled_vertex]{};
\end{tikzpicture}}  }
\newcommand{\Xb}{\vc{\begin{tikzpicture}\outercycle{4}{3}
\draw[edge_color1] (x0)--(x1);\draw[edge_color1] (x0)--(x2);  \draw[edge_color2] (x1)--(x2);    
\draw (x0) node[labeled_vertex]{};\draw (x1) node[unlabeled_vertex]{};\draw (x2) node[unlabeled_vertex]{};
\end{tikzpicture}}   }
\newcommand{\Xc}{\vc{\begin{tikzpicture}\outercycle{4}{3}
\draw[edge_color1] (x0)--(x1);\draw[edge_color2] (x0)--(x2);  \draw[edge_color1] (x1)--(x2);    
\draw (x0) node[labeled_vertex]{};\draw (x1) node[unlabeled_vertex]{};\draw (x2) node[unlabeled_vertex]{};
\end{tikzpicture}}   }
\newcommand{\Xd}{\vc{\begin{tikzpicture}\outercycle{4}{3}
\draw[edge_color1] (x0)--(x1);\draw[edge_color2] (x0)--(x2);  \draw[edge_color2] (x1)--(x2);    
\draw (x0) node[labeled_vertex]{};\draw (x1) node[unlabeled_vertex]{};\draw (x2) node[unlabeled_vertex]{};
\end{tikzpicture}}   }
\newcommand{\Xe}{\vc{\begin{tikzpicture}\outercycle{4}{3}
\draw[edge_color2] (x0)--(x1);\draw[edge_color2] (x0)--(x2);  \draw[edge_color1] (x1)--(x2);    
\draw (x0) node[labeled_vertex]{};\draw (x1) node[unlabeled_vertex]{};\draw (x2) node[unlabeled_vertex]{};
\end{tikzpicture}}   }
\newcommand{\Xf}{\vc{\begin{tikzpicture}\outercycle{4}{3}
\draw[edge_color2] (x0)--(x1);\draw[edge_color2] (x0)--(x2);  \draw[edge_color2] (x1)--(x2);    
\draw (x0) node[labeled_vertex]{};\draw (x1) node[unlabeled_vertex]{};\draw (x2) node[unlabeled_vertex]{};
\end{tikzpicture}}}
%
%

In Appendix~\ref{appendix:A} we list all coefficients for products in $\mathcal{F}^1_3$, after unlabeling and multiplying by a scaling factor of 30 to clear denominators. 
Then we obtain that
\[
\llbracket X^T M X \rrbracket_1 = \sum_{i=1}^6 \sum_{j=1}^6 M_{i,j} \llbracket X_i \times X_j \rrbracket_1 = \sum_{F \in \mathcal{F}_5} c_F^M \cdot F, 
\]
since each $\llbracket X_i \times X_j \rrbracket_1$ is a linear combination of graphs in $\mathcal{F}_5$.

\subsubsection{Guessing matrices $A$ and $B$}\label{subsec:guessing}

In this paragraph we describe how we obtained the matrices $A$ and $B$.
First, we used semidefinite programming to find a matrix $M$ for several small odd values of $k$.
Notice that if \eqref{eq:FA} is applied to the extremal construction, then
the left-hand side is equal to the right-hand side. That means that all inequalities used are actually equalities.
In particular, $\phi\left(  \llbracket X^T M X \rrbracket_1  \right) = 0$. Since $M$ is a positive semidefinite
matrix, $X$ evaluated on our extremal example (the limit of $T_k^n$ as $n\to\infty$) must give an eigenvector of $M$ corresponding
to the eigenvalue 0. The matrix $B$ was obtained by projecting onto the space orthogonal to three zero
eigenvectors of  $M$. As noted before, we had one zero eigenvector to start with.
By looking at all eigenvectors of $M$, we managed to guess another zero eigenvector. 
We tried projection with the two zero eigenvectors and found the third one in the projection. 
After having obtained matrices $B$, we observed that a suitable $A$ exists even if we set the coordinate $[1,2]$ and $[2,1]$ to 0. With proper scaling of the objective function, we were getting nice matrices from the CSDP~\cite{Borchers1999} solver with all entries
integers. By using the solutions for several values of $k$, we calculated a polynomial function of $k$ fitting each 
entry in matrix $A$. Finally we observed that the same matrices $A$ and $B$ also work for even values of $k$.

\section{Stability}\label{sec:stability}

In this section we prove Theorem~\ref{th:stab}. For this purpose it will be convenient to rewrite our lower bound  as an asymptotic identity valid for an arbitrary graph. Fix $k\ge 3$. Let $p$ and $\lambda$ be as in~\eqref{eq:p}. Let the matrix $M\succeq 0$, $\alpha\in\I R
$, and the reals $c_F^M,c_F^{OPT},c_F$, indexed by $F\in\C F_5$, be as previously.

Recall that $X=(X_1,\dots,X_6)^T$ is the vector of $3$-vertex rooted graphs defined in~\eqref{eq:X}. For a graph $G=(V,E)$ of order $n\ge 5$ and a vertex $r\in V$, let $Y_r$ be the column vector whose $i$-th component is the number of unordered 2-sets $\{u,v\}\subseteq V\setminus\{r\}$ such that the induced graph $G[\{r,u,v\}]$ rooted at $r$ is isomorphic to $X_i$.  Define 
$
\O{Y}=\frac4{5!}\,\sum_{r\in V} Y_r^TMY_r\ge 0.
$ 

Let us argue that
 \begin{equation}\label{eq:sigma1}
 \O{Y}=\sum_{F\in\C F_5} c_F^M P(F,G)+O(n^4),
 \end{equation}
 where for $F\in\C F_\ell$ we let $P(F,G)={n\choose \ell} p(F,G)$ be the number of $\ell$-sets inducing a copy of $F$ in $G$. Indeed, the $i$-th entry of $Y_r$ can be  written as a double sum $\frac12 \sum_{u\in V}\sum_{v\in V}$ of the indicator function that $r,u,v$ are distinct and the graph $G[\{r,u,v\}]$ when rooted at $r$ is isomorphic to $X_i$. Using this representation of $Y_r$ and expanding everything, we can write $\O{Y}$ as a sum over all $(r,u,v,u',v')\in V^5$ of some function that depends only on the graph induced by the (multi)set $(r,u,v,u',v')$ inside $G$. Apart of $O(n^4)$ terms when some of the vertices coincide, the remaining ones can be grouped by the isomorphism type $F\in\C F_5$ of $G[\{r,u,v,u',v'\}]$. For $F\in\C F_5$, each unordered $5$-set spanning an induced copy of $F$ in $G$ contributes the same amount (depending only on $F$ and $M$) and the coefficient $c_F^M$ was in fact defined by us to be equal to this common value. Thus~\eqref{eq:sigma1} holds.

Likewise, $P(K_2,G){n-2\choose 3}$ and ${n\choose 5}$ can be written as fixed linear combinations of $P(F,G)$ over $F\in\C F_5$. Also, $d_{C_5}(G)n^5=\sum_{F\in\C F_5} c_F^{OPT} P(F,G)$ is the number of $5$-cycles in $G$. Putting all together, we obtain the following identity valid for an arbitrary graph $G$:
 \begin{equation}\label{eq:main}
 d_{C_5}(G)\,n^5 +\frac{\alpha}{5!}\left(2P(K_2,G)n^3-pn^5\right)-\O{Y}+O(n^4)
 =\sum_{F\in\C F_5} c_F P(F,G),
 \end{equation}
  where $c_F$ for $F\in\C F_5$ was defined to be exactly the contribution of each induced copy of $F$ in $G$ to the left-hand side while  all combinations when some vertices in the underlying 5-fold sum coincide are absorbed into the error term~$O(n^4)$.

 Note that if we multiply~\eqref{eq:main} by ${n\choose 5}^{-1}$ then the scaled terms in~\eqref{eq:main} will be asymptotically the same as in~\eqref{eq:FA} when $n\to\infty$. Since $\sum_{F\in\C F_5} P(F,G)={n\choose 5}$, the right-hand size of~\eqref{eq:main} can be lower bounded by ${n\choose 5}\min_{F\in\C F_5} c_F$, giving the required lower bound in Theorem~\ref{thm:main} since each $c_F$ is at least $5!\, \lambda$.

Let us turn to stability. Take any sequence of graphs $G_m$ of strictly increasing orders 
such that 
\begin{equation}\label{eq:Gm}
|E(G_m)|\ge \left(p-\frac 1m\right){|G_m|\choose 2}\quad\mbox{and}\quad d_{C_5}(G_{m})\le \lambda+\frac1m,\qquad \mbox{for all }m\in\I N.
\end{equation}

Observe that if $c_F > \lambda$ for some $F\in\C F_5$, then the right-hand side of~\eqref{eq:main} is at least $
(\lambda+(c_F-\lambda)p(F,G)){n\choose 5}.
$ 
Thus we have that $p(F,G_m)=o(1)$ as $m\to\infty$ for every such $F$. By looking at the explicit formulas for $c_F$ near the end of Section~\ref{subsec:nice_p}, we see that there are 16 such graphs. They are collected into the list $\mathcal L$ in Figure~\ref{fg:L}, and are denoted by $L_1,\dots,L_{16}$ in this order.

\begin{figure}[h]
\begin{center}
\begin{tikzpicture}\outercycle{6}{5}
\draw[edge_color1] (x0)--(x1);\draw[edge_color1] (x0)--(x2);\draw[edge_color1] (x0)--(x3);\draw[edge_color1] (x0)--(x4);  \draw[edge_color1] (x1)--(x2);\draw[edge_color1] (x1)--(x3);\draw[edge_color1] (x1)--(x4);  \draw[edge_color1] (x2)--(x3);\draw[edge_color1] (x2)--(x4);  \draw[edge_color2] (x3)--(x4);    
\draw (x0) node[unlabeled_vertex]{};\draw (x1) node[unlabeled_vertex]{};\draw (x2) node[unlabeled_vertex]{};\draw (x3) node[unlabeled_vertex]{};\draw (x4) node[unlabeled_vertex]{};
\end{tikzpicture} 
\begin{tikzpicture}\outercycle{6}{5}
\draw[edge_color1] (x0)--(x1);\draw[edge_color1] (x0)--(x2);\draw[edge_color1] (x0)--(x3);\draw[edge_color1] (x0)--(x4);  \draw[edge_color1] (x1)--(x2);\draw[edge_color1] (x1)--(x3);\draw[edge_color1] (x1)--(x4);  \draw[edge_color1] (x2)--(x3);\draw[edge_color2] (x2)--(x4);  \draw[edge_color2] (x3)--(x4);    
\draw (x0) node[unlabeled_vertex]{};\draw (x1) node[unlabeled_vertex]{};\draw (x2) node[unlabeled_vertex]{};\draw (x3) node[unlabeled_vertex]{};\draw (x4) node[unlabeled_vertex]{};
\end{tikzpicture} 
\begin{tikzpicture}\outercycle{6}{5}
\draw[edge_color1] (x0)--(x1);\draw[edge_color1] (x0)--(x2);\draw[edge_color1] (x0)--(x3);\draw[edge_color1] (x0)--(x4);  \draw[edge_color1] (x1)--(x2);\draw[edge_color1] (x1)--(x3);\draw[edge_color2] (x1)--(x4);  \draw[edge_color1] (x2)--(x3);\draw[edge_color2] (x2)--(x4);  \draw[edge_color2] (x3)--(x4);    
\draw (x0) node[unlabeled_vertex]{};\draw (x1) node[unlabeled_vertex]{};\draw (x2) node[unlabeled_vertex]{};\draw (x3) node[unlabeled_vertex]{};\draw (x4) node[unlabeled_vertex]{};
\end{tikzpicture} 
\begin{tikzpicture}\outercycle{6}{5}
\draw[edge_color1] (x0)--(x1);\draw[edge_color1] (x0)--(x2);\draw[edge_color1] (x0)--(x3);\draw[edge_color1] (x0)--(x4);  \draw[edge_color1] (x1)--(x2);\draw[edge_color1] (x1)--(x3);\draw[edge_color1] (x1)--(x4);  \draw[edge_color2] (x2)--(x3);\draw[edge_color2] (x2)--(x4);  \draw[edge_color2] (x3)--(x4);    
\draw (x0) node[unlabeled_vertex]{};\draw (x1) node[unlabeled_vertex]{};\draw (x2) node[unlabeled_vertex]{};\draw (x3) node[unlabeled_vertex]{};\draw (x4) node[unlabeled_vertex]{};
\end{tikzpicture} 
\begin{tikzpicture}\outercycle{6}{5}
\draw[edge_color1] (x0)--(x1);\draw[edge_color1] (x0)--(x2);\draw[edge_color1] (x0)--(x3);\draw[edge_color1] (x0)--(x4);  \draw[edge_color1] (x1)--(x2);\draw[edge_color1] (x1)--(x3);\draw[edge_color2] (x1)--(x4);  \draw[edge_color2] (x2)--(x3);\draw[edge_color1] (x2)--(x4);  \draw[edge_color1] (x3)--(x4);    
\draw (x0) node[unlabeled_vertex]{};\draw (x1) node[unlabeled_vertex]{};\draw (x2) node[unlabeled_vertex]{};\draw (x3) node[unlabeled_vertex]{};\draw (x4) node[unlabeled_vertex]{};
\end{tikzpicture} 
\begin{tikzpicture}\outercycle{6}{5}
\draw[edge_color1] (x0)--(x1);\draw[edge_color1] (x0)--(x2);\draw[edge_color1] (x0)--(x3);\draw[edge_color1] (x0)--(x4);  \draw[edge_color1] (x1)--(x2);\draw[edge_color1] (x1)--(x3);\draw[edge_color2] (x1)--(x4);  \draw[edge_color2] (x2)--(x3);\draw[edge_color1] (x2)--(x4);  \draw[edge_color2] (x3)--(x4);    
\draw (x0) node[unlabeled_vertex]{};\draw (x1) node[unlabeled_vertex]{};\draw (x2) node[unlabeled_vertex]{};\draw (x3) node[unlabeled_vertex]{};\draw (x4) node[unlabeled_vertex]{};
\end{tikzpicture} 
\begin{tikzpicture}\outercycle{6}{5}
\draw[edge_color1] (x0)--(x1);\draw[edge_color1] (x0)--(x2);\draw[edge_color1] (x0)--(x3);\draw[edge_color1] (x0)--(x4);  \draw[edge_color1] (x1)--(x2);\draw[edge_color1] (x1)--(x3);\draw[edge_color2] (x1)--(x4);  \draw[edge_color2] (x2)--(x3);\draw[edge_color2] (x2)--(x4);  \draw[edge_color2] (x3)--(x4);    
\draw (x0) node[unlabeled_vertex]{};\draw (x1) node[unlabeled_vertex]{};\draw (x2) node[unlabeled_vertex]{};\draw (x3) node[unlabeled_vertex]{};\draw (x4) node[unlabeled_vertex]{};
\end{tikzpicture} 
\begin{tikzpicture}\outercycle{6}{5}
\draw[edge_color1] (x0)--(x1);\draw[edge_color1] (x0)--(x2);\draw[edge_color1] (x0)--(x3);\draw[edge_color2] (x0)--(x4);  \draw[edge_color1] (x1)--(x2);\draw[edge_color1] (x1)--(x3);\draw[edge_color2] (x1)--(x4);  \draw[edge_color2] (x2)--(x3);\draw[edge_color1] (x2)--(x4);  \draw[edge_color1] (x3)--(x4);    
\draw (x0) node[unlabeled_vertex]{};\draw (x1) node[unlabeled_vertex]{};\draw (x2) node[unlabeled_vertex]{};\draw (x3) node[unlabeled_vertex]{};\draw (x4) node[unlabeled_vertex]{};
\end{tikzpicture}\\[2mm] 
\begin{tikzpicture}\outercycle{6}{5}
\draw[edge_color1] (x0)--(x1);\draw[edge_color1] (x0)--(x2);\draw[edge_color1] (x0)--(x3);\draw[edge_color2] (x0)--(x4);  \draw[edge_color1] (x1)--(x2);\draw[edge_color1] (x1)--(x3);\draw[edge_color2] (x1)--(x4);  \draw[edge_color2] (x2)--(x3);\draw[edge_color1] (x2)--(x4);  \draw[edge_color2] (x3)--(x4);    
\draw (x0) node[unlabeled_vertex]{};\draw (x1) node[unlabeled_vertex]{};\draw (x2) node[unlabeled_vertex]{};\draw (x3) node[unlabeled_vertex]{};\draw (x4) node[unlabeled_vertex]{};
\end{tikzpicture} 
\begin{tikzpicture}\outercycle{6}{5}
\draw[edge_color1] (x0)--(x1);\draw[edge_color1] (x0)--(x2);\draw[edge_color1] (x0)--(x3);\draw[edge_color1] (x0)--(x4);  \draw[edge_color1] (x1)--(x2);\draw[edge_color2] (x1)--(x3);\draw[edge_color2] (x1)--(x4);  \draw[edge_color2] (x2)--(x3);\draw[edge_color2] (x2)--(x4);  \draw[edge_color1] (x3)--(x4);    
\draw (x0) node[unlabeled_vertex]{};\draw (x1) node[unlabeled_vertex]{};\draw (x2) node[unlabeled_vertex]{};\draw (x3) node[unlabeled_vertex]{};\draw (x4) node[unlabeled_vertex]{};
\end{tikzpicture} 
\begin{tikzpicture}\outercycle{6}{5}
\draw[edge_color1] (x0)--(x1);\draw[edge_color1] (x0)--(x2);\draw[edge_color1] (x0)--(x3);\draw[edge_color2] (x0)--(x4);  \draw[edge_color1] (x1)--(x2);\draw[edge_color2] (x1)--(x3);\draw[edge_color1] (x1)--(x4);  \draw[edge_color2] (x2)--(x3);\draw[edge_color2] (x2)--(x4);  \draw[edge_color1] (x3)--(x4);    
\draw (x0) node[unlabeled_vertex]{};\draw (x1) node[unlabeled_vertex]{};\draw (x2) node[unlabeled_vertex]{};\draw (x3) node[unlabeled_vertex]{};\draw (x4) node[unlabeled_vertex]{};
\end{tikzpicture} 
\begin{tikzpicture}\outercycle{6}{5}
\draw[edge_color1] (x0)--(x1);\draw[edge_color1] (x0)--(x2);\draw[edge_color1] (x0)--(x3);\draw[edge_color2] (x0)--(x4);  \draw[edge_color1] (x1)--(x2);\draw[edge_color2] (x1)--(x3);\draw[edge_color1] (x1)--(x4);  \draw[edge_color2] (x2)--(x3);\draw[edge_color2] (x2)--(x4);  \draw[edge_color2] (x3)--(x4);    
\draw (x0) node[unlabeled_vertex]{};\draw (x1) node[unlabeled_vertex]{};\draw (x2) node[unlabeled_vertex]{};\draw (x3) node[unlabeled_vertex]{};\draw (x4) node[unlabeled_vertex]{};
\end{tikzpicture} 
\begin{tikzpicture}\outercycle{6}{5}
\draw[edge_color1] (x0)--(x1);\draw[edge_color1] (x0)--(x2);\draw[edge_color1] (x0)--(x3);\draw[edge_color2] (x0)--(x4);  \draw[edge_color1] (x1)--(x2);\draw[edge_color2] (x1)--(x3);\draw[edge_color2] (x1)--(x4);  \draw[edge_color2] (x2)--(x3);\draw[edge_color2] (x2)--(x4);  \draw[edge_color1] (x3)--(x4);    
\draw (x0) node[unlabeled_vertex]{};\draw (x1) node[unlabeled_vertex]{};\draw (x2) node[unlabeled_vertex]{};\draw (x3) node[unlabeled_vertex]{};\draw (x4) node[unlabeled_vertex]{};
\end{tikzpicture} 
\begin{tikzpicture}\outercycle{6}{5}
\draw[edge_color1] (x0)--(x1);\draw[edge_color1] (x0)--(x2);\draw[edge_color1] (x0)--(x3);\draw[edge_color2] (x0)--(x4);  \draw[edge_color2] (x1)--(x2);\draw[edge_color2] (x1)--(x3);\draw[edge_color1] (x1)--(x4);  \draw[edge_color2] (x2)--(x3);\draw[edge_color1] (x2)--(x4);  \draw[edge_color1] (x3)--(x4);    
\draw (x0) node[unlabeled_vertex]{};\draw (x1) node[unlabeled_vertex]{};\draw (x2) node[unlabeled_vertex]{};\draw (x3) node[unlabeled_vertex]{};\draw (x4) node[unlabeled_vertex]{};
\end{tikzpicture} 
\begin{tikzpicture}\outercycle{6}{5}
\draw[edge_color1] (x0)--(x1);\draw[edge_color1] (x0)--(x2);\draw[edge_color1] (x0)--(x3);\draw[edge_color2] (x0)--(x4);  \draw[edge_color2] (x1)--(x2);\draw[edge_color2] (x1)--(x3);\draw[edge_color1] (x1)--(x4);  \draw[edge_color2] (x2)--(x3);\draw[edge_color1] (x2)--(x4);  \draw[edge_color2] (x3)--(x4);    
\draw (x0) node[unlabeled_vertex]{};\draw (x1) node[unlabeled_vertex]{};\draw (x2) node[unlabeled_vertex]{};\draw (x3) node[unlabeled_vertex]{};\draw (x4) node[unlabeled_vertex]{};
\end{tikzpicture} 
\begin{tikzpicture}\outercycle{6}{5}
\draw[edge_color1] (x0)--(x1);\draw[edge_color1] (x0)--(x2);\draw[edge_color2] (x0)--(x3);\draw[edge_color2] (x0)--(x4);  \draw[edge_color2] (x1)--(x2);\draw[edge_color1] (x1)--(x3);\draw[edge_color2] (x1)--(x4);  \draw[edge_color2] (x2)--(x3);\draw[edge_color1] (x2)--(x4);  \draw[edge_color2] (x3)--(x4);    
\draw (x0) node[unlabeled_vertex]{};\draw (x1) node[unlabeled_vertex]{};\draw (x2) node[unlabeled_vertex]{};\draw (x3) node[unlabeled_vertex]{};\draw (x4) node[unlabeled_vertex]{};
\end{tikzpicture}
\end{center}
\caption{The list $\mathcal L=(L_1,\dots,L_{16})$}\label{fg:L}
\end{figure}
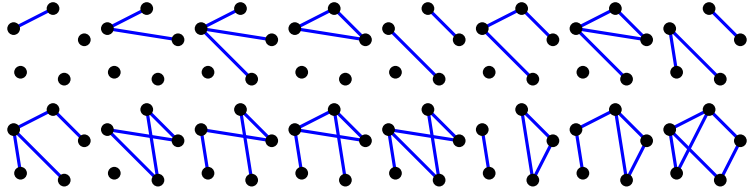


Let the \emph{co-cherry} $\O{P_2}$ be the complement of the 2-edge path $P_2$, that is, $\O{P_2}$ is the graph with 3 vertices and 1 edge. Next, we show that its density in $G_m$ must also be $o(1)$. Note that there are 5-vertex graphs not in the list $\C L$ that contain the co-cherry. Thus the naive approach does not work and a slightly more involved argument is needed.

\begin{lemma}\label{lm:NoCocherries} For every sequence  of graphs $G_m$ as in~\eqref{eq:Gm}, we have that 
	$$
	\lim_{m\to\infty} p(\O{P_2},G_m)=0.
	$$
\end{lemma}

\begin{proof} Let $m$ be sufficiently large, $G=G_m$, $V=V(G)$ and $n=|V|$.
For $i\in\{0,1,2\}$, let $F_i$ be the (unique) graph of order 4 with $i$ disjoint edges. Let $\C L'=\C L\cup\{F_0,F_1,F_2\}$. 

Apply the induced removal lemma (see, e.g.,~\cite{Alon2000,ConFox2013}) to $G$
to destroy  all induced graphs in $\C L'$ whose density is $o(1)$.
Formally, let $f=n^{-1}+\max\{ p(L,G_m)\mid L\in\C L\}$ (and let initially $G=G_m$). As long as there is at least one $F\in\C L'$ with $0<p(F,G)\le f$, change as few as possible adjacencies in $G$ to destroy all copies of all such $F$ so that, additionally, no graph in $\C L'$ absent from $G$ is introduced.
Since $f$ tends to $0$ as $m\to\infty$ and the above iteration is applied at most $|\C L'|$ times (in fact, at most $|\C L'\setminus \C L|+1=4$ times), we change $o(n^2)$ edges in total by the induced removal lemma. Also, the final graph $G$ contains no graph from the list $\C L$ since the first iteration destroyed all such subgraphs by our choice of~$f$.

\begin{claim}\label{cl:NoF1} $G$ contains no induced $F_1$ (i.e.,\ 4 vertices spanning exactly one edge).\end{claim}

\begin{proof} Take a copy of $F_1$ and add one new vertex $x$ of degree $d$. If $d\in\{0,1,2,3\}$, then the sets of possible obtained graphs up to isomorphism are respectively $\{L_1\}$, $\{L_2,L_5\}$, $\{L_4,L_6,L_8\}$, and $\{L_7,L_9\}$. We see that each 1-vertex extension of $F_1$ is in $\C L$ except when $d=4$ (that is, when $x$ is adjacent to every vertex of $F_1$). This means that for every copy of $F_1$, say on $A\subseteq V$, the set $A$ is complete
to $V\setminus A$ in~$G$. It follows that every two distinct induced copies of $F_1$ are vertex-disjoint and thus $G$ has at most $n/4$ such copies. This is at most $f{n\choose 4}$, so $G$ has no copy of $F_1$ at all.\end{proof}

\begin{claim} $G$ contains no induced $F_2$ (which is the matching with two edges).\end{claim}
\begin{proof} If we extend a copy of $F_2$ by adding a vertex $x$ of degree $d\in\{0,1,2,3\}$, then we obtain graphs in respectively $\{L_5\}$, $\{L_8\}$, $\{L_{11},L_{14}\}$ and $\{L_{15}\}$. Thus the only extension that does not lead to a graph in $\C L$ is to connect $x$ to every vertex of~$F_2$. This gives that every two distinct induced copies of $F_2$ in $G$ are vertex-disjoint. Thus we have at most $O(n) \le f{n\choose 4}$ copies of $F_2$, that is, none at all.\end{proof}

Consider the edgeless 4-vertex graph $F_0$. If we add a vertex $x$ of degree $d\in\{1,2,3\}$, then we get respectively $L_1$, $L_2$ and $L_3$. The only remaining ways are to have $x$ empty or complete to $F_0$. Now, consider any copy of $F_0$ in $G$, say with vertex set $A_0\subseteq V(G)$. By above, every vertex outside of $A_0$ is empty or complete to~$A_0$. Let $A\supseteq A_0$ consist of all vertices of $G$ that send no edges to $A_0$. Note that $A$ is an independent set: if we had an edge $xy$ inside $A$ then $x$, $y$ plus some two extra vertices from $A_0$ would span a copy of $F_1$ in $G$, contradicting Claim~\ref{cl:NoF1}. Moreover, $A$ is complete to $V\setminus A$. Indeed, for every pair $(a,b)\in A\times (V\setminus A)$, the subgraph of $G$ induced by $a$ and some further three vertices of $A_0$  has no edges; thus the vertex $b\not\in A$ must be complete to it.

It follows that we can find disjoint independent sets $A_i$, $i\in I$, in $V$ such that each $A_i$ 
is complete to $V\setminus A_i$ while every copy of $F_0$ in $G$ is inside one of these sets~$A_i$. 
Define $B=V\setminus(\cup_{i\in I} A_i)$.

By the definition of $B$ and the above claims, we have that $H=G[B]$ is $\{F_0,F_1,F_2\}$-free. This means that the complement $\O H$ of $H$ cannot have a (not necessarily induced) 4-cycle $C_4$ because for any way of filling its diagonals we get $F_0$, $F_1$ or $F_2$ in $H$. Thus $|E(\O H)|$ is at most the Tur\'an function $\operatorname{ex}(n,C_4)=O(n^{3/2})$, that is, $H$ is $O(n^{3/2})$-close in the edit distance to being a complete graph. We see that $G$ is $O(n^{3/2})$-close to the complete partite graph $G'$ with parts $A_i$, $i\in I$, and $\{x\}$, $x\in B$. As  every co-cherry in $G$ has to contain at least one pair where $E(G)$ and $E(G')$ differ, $G$ has at most $O(n^{5/2})$ co-cherries. 

Since the original graph $G_m$ and $G$ differ in $o(n^2)$ adjacencies, the co-cherry density in $G_m$ is $o(1)$, as required.\end{proof}

Thus, 
another application of the induced removal lemma gives that we can change $o(1)$-fraction of adjacencies in $G_m$ and make it $\O{P_2}$-free, that is, complete partite.
Thus, in order to finish the proof of Theorem \ref{th:stab}, it is enough to argue that each of the $k$ largest parts of $G_m$ has $(\frac1k+o(1))|G_m|$ vertices. We present two proofs of this. The first proof is more direct but longer. The second one is shorter but assumes some known facts about graphons.

\subsection{First proof}

We need the following auxiliary result.

\begin{lemma}\label{lem:tech}
Suppose a graph $J$ on $n$ vertices has a subgraph $X$ such that 
\begin{enumerate}[(i)]
\item \label{cond1} $X$ has $x$ vertices where $\eps' n \le x \le (1-\eps')n$ and edge density $q \le \frac{1}{2}$
\item \label{cond2}$X$ is complete to $V(J)\setminus X$
\item \label{cond3}$X$ contains at least $\frac 12 x^4 q^3 + \eps' x^4$ copies of $P_4$.
\end{enumerate}
Then there exists a graph $J'$ on $n$ vertices with asymptotically the same edge density as $J$ and $$d_{C_5}(J') \le d_{C_5}(J) - \frac 12 (\eps')^6.$$
\end{lemma}

\begin{proof}
Note first that conditions \eqref{cond1} and \eqref{cond2} imply that $J$ is dense since it has at least $\eps' (1-\eps')n^2$ edges. We make $J'$ by replacing $X$ with a $X'$, which is a random balanced bipartite graph with edge probability $2q$. We will not change the rest of the graph, so $J'-X'=J-X$. W.h.p.  $X'$ has edge density asymptotically $q$ and so $J'$ has asymptotically the same edge density as $J$. We will argue that $J'$ has much fewer copies of $C_5$ than $J$ has, by considering several possible types of $C_5$ copies. 

We will compare the copies according to how they intersect $X$ (for counting copies of $C_5$ in the graph $J$) or $X'$ (in $J'$). Specifically, since $X$  is complete to the rest of  $J$   we have
\[
\nu_{C_5}(J) = \sum_H m_H \nu_{H}(X)\cdot \nu_{C_5-H}(J - X)
\]
where the sum is over all induced subgraphs $H \subseteq C_5$, and the coefficient $m_H$ is the number of $C_5$ copies contained in the graph formed by taking a copy of $H$ and a copy of $C_5-H$ with every possible edge in between. Recall that $\nu_H(G)$ counts the number of (not necessarily induced) copies of $H$ in $G$. Similarly,  we have
\[
\nu_{C_5}(J')= \sum_H m_H \nu_{H}(X')\cdot \nu_{C_5-H}(J'-X') =\sum_H m_H \nu_{H}(X')\cdot \nu_{C_5-H}(J-X),
\]
since $J'-X' = J-X$. So we will compare $\nu_{H}(X)$ with $\nu_{H}(X')$ for each $H$. Specifically we will show that $\nu_{H}(X') \le (1+o(1))\nu_{H}(X)$ for each $H$, and that this inequality holds with some room for $H=P_4$.

Some easy cases: when $H$ has no vertices, $\nu_{H}(X) = \nu_{H}(X') = 1$. 
When $H$ is a single vertex, $\nu_{H}(X) = \nu_{H}(X') = x$. When $H$ is just an edge, $\nu_{H}(X) = (1+o(1))\nu_{H}(X') = (1+o(1))\binom{x}{2}q$. When $H$ has 2 vertices and no edge we have $\nu_{H}(X') = \nu_{H}(X) = \binom{x}{2}$.  When $H$ is the graph on 3 vertices consisting of an edge and an isolated vertex, we have $\nu_{H}(X') = (1+o(1))\nu_{H}(X) = (1+o(1))x\binom{x}{2}q$.

When $H=P_3$ (the path of length 2) we have 
\[
\nu_{P_3}(X')= 2\binom{\frac x2}{2}\frac x2(2q)^2 = (1+o(1))\frac12 x^3 q^2
\]
 which we compare to 
 \[
\nu_{P_3}(X)= \sum_{v \in X} \binom{|N(v) \cap X|}{2} \ge x \cdot \binom{\frac{2q \binom{x}{2}}{x}}{2}=(1+o(1))\frac12 x^3 q^2.
 \]

Finally we consider the case $H=P_4$. We have
\[
\nu_{P_4}(X')=2\binom{\frac x2}{2} \cdot 2\binom{\frac x2}{2} (2q)^3 = (1+o(1)) \frac 1{2} x^4 q^3
\]
which we compare to
 \[
\nu_{P_4}(X)= \frac 12 x^4 q^3 + \eps' x^4.
 \]

Taking all possible $H$ into account, we see that 
\begin{align*}
\nu_{C_5}(J)- \nu_{C_5}(J') &=  \sum_H \sbrac{ \nu_{H}(X) - \nu_{H}(X')}\cdot \nu_{C_5-H}(J - X) \\
& \ge \sbrac{\nu_{P_4}(X) - \nu_{P_4}(X')}\cdot\nu_{C_5-P_4}(J - X) \\
& \ge (1+o(1)) \eps' x^4 \cdot (n-x) \\
& >\frac 12 (\eps')^6 n^5 
\end{align*}
and so $$d_{C_5}(J') \le d_{C_5}(J) - \frac 12 (\eps')^6.$$
\end{proof}

\begin{proof}[Proof of Theorem \ref{th:stab}] Let $G_m$ be as in~\eqref{eq:Gm}. Let $m\to\infty$.
By the induced graph removal lemma and Lemma~\ref{lm:NoCocherries}
we can eliminate all co-cherries in the graph $G=G_m$ of order $n\to\infty$ by adding or removing at most $\a n^2$ edges, for some $\a = \a(\eps) \rightarrow 0$  as $\eps\rightarrow 0$. Call this new graph $G'$, which has edge density $p'$, where $p - 2\a \leq p' \leq p+2\a$.  Moreover, $G'$ is a complete $k'$-partite graph for some~$k'$. Say the parts of $G'$ are $X_1, \ldots,X_{k'}$.
Also, note that since adding (or removing) one edge to $G$ creates (or destroys) at most $n^3$ copies of $C_5$, we have
\[
d_{C_5} (G) = d_{C_5} (G') + O(\a),
\] 
and 
\[
d_{C_5}(p)  =  d_{C_5}(p')+O(\a)
\]
(recall that we use big-O notation to replace quantities that are bounded in absolute value, and the quantity being replaced may be negative). Now 
\begin{equation}\label{eq:dupper}
d_{C_5}(G') \le d_{C_5}(G) + O(\a) \le d_{C_5}(p)+ \eps  + O(\a) \le d_{C_5}(p')+  O(\eps + \a)
\end{equation}
and so $G'$ has nearly the minimum  $C_5$-density among graphs with edge density $p'$. 

In the following, we will need  a parameter $\b = \b(\eps) = (\eps+\a(\eps))^{1/100}$. 

\begin{claim}
We are done unless we have the following. For any $i \neq j$, $|X_i|+|X_j| \le (1-\b)n$.
\end{claim}
\begin{proof}
Without loss of generality, suppose for contradiction that  $|X_1| + |X_2| \ge (1-\b)n$, so the number of edges in $G'$ is at most 
\[
\binom{n}{2} - \binom{|X_1|}{2} - \binom{|X_2|}{2} \le \binom{n}{2}- 2\binom{\frac{(1-\b)n}{2}}{2} \le \frac12 n^2 - \frac14 (1-\b)^2 n^2= \rbrac{\frac14 +O(\b)} n^2
\]
 and so we must have $k=2$ since throughout the proof we assume $\eps$ (and therefore $\a$ and $\b$) are sufficiently small. Now if $||X_1| - |X_2|| >\b^{1/3} n$, say without loss of generality $|X_1| > |X_2| + \b^{1/3} n$ then the number of edges in $G'$ is at most 
 \begin{align*}
|X_1||X_2| + \b n (|X_1| + |X_2|) + \binom{\b n }{2} &\le \rbrac{\frac n2 +  \frac12 \b^{1/3} n}\rbrac{\frac n2 -  \frac12 \b^{1/3}  n} + \b n^2 + \binom{\b n }{2}\\
& = \rbrac{\frac14 -\frac14 \b^{2/3} + O(\b)}n^2,
\end{align*}
 which is a contradiction for small $\eps$ since $G'$ has at least $\binom{n}{2}p- \a n^2$ edges (where $p=\frac12$ since $k=2$) and $\frac 14 \b^{2/3} +O(\b) > \a$ for small $\eps$. To summarize, $G'$ is a complete partite graph that has two large parts $X_1, X_2$ which differ in size by at most $\b^{1/3} n$, and together the rest of the parts make up at most $\b n$ vertices. It is easy to see then that $G'$ can be changed into a balanced complete bipartite graph by editing $O(\b^{1/3} n^2)$ edges.  
\end{proof}
Thus, we henceforth assume that for any $i \neq j$, $|X_i|+|X_j| \le (1-\b)n$.

\begin{claim} \label{claim:balanced}
For all $i, j$, if $|X_i|, |X_j| \ge \b n$, then $||X_i| - |X_j|| \le \b n$.
\end{claim}
\begin{proof}
Suppose for contradiction that there are two parts (without loss of generality say $X_1, X_2$) such that $|X_1|, |X_2| \ge \b n$ and $||X_1| - |X_2|| > \b n$. 
We will derive a contradiction by arguing that $G'$ can be modified by Lemma \ref{lem:tech} to form another graph $G^*$ of asymptotically the same edge density but with significantly smaller $C_5$-density than $G'$. 
%

We apply Lemma \ref{lem:tech} with $J=G'$, $X=X_1 \cup X_2$, $\eps' = \frac12 \b^6$ and $$q = \frac{x_1 x_2}{\binom{x}{2}} = (1+o(1))\frac{2 x_1 x_2}{x^2} $$ where $|X_i|=x_i$ and $x=x_1+x_2$. 
Let us check the conditions of the lemma. Clearly we have 
\[
\b n \le x \le (1-\b)n,
\]
and $X$ is complete to the rest of the graph (since $X$ is composed of two parts of a complete partite graph). Finally, the number of copies of $P_4$ in $X$ is 
\begin{align*}
\nu_{P_4}(X) = 2 \binom{x_1}{2} \cdot 2 \binom{x_2}{2} = (1+o(1))x_1^2 x_2^2
\end{align*}
which we compare to 
\begin{align*}
\frac 12 x^4 q^3 = (1+o(1)) \frac 12 x^4 \rbrac{\frac{2 x_1 x_2}{x^2}}^3= (1+o(1)) \frac{4x_1^3 x_2^3}{x^2}.
\end{align*}
From here we can see that 
\begin{align*}
\nu_{P_4}(X) - \frac 12 x^4 q^3 & \ge (1+o(1))\rbrac{x_1^2 x_2^2 - \frac{4x_1^3 x_2^3}{x^2}} \\
& \ge \frac12 \cdot \frac{x_1^2x_2^2}{x^2} (x^2 - 4x_1x_2) \\
& = \frac12 \cdot \frac{x_1^2x_2^2}{x^2} (x_1-x_2)^2 \\
&\ge \frac12 \frac{(\b n)^4}{n^2} (\b n)^2 = \frac12 \b^6 n^4 \ge \frac12 \b^6 x^4
\end{align*}
and so Lemma \ref{lem:tech} applies, implying that $J=G'$ must have $C_5$-density at least $$d_{C_5}(p') + \frac12 \rbrac{\frac12 \b^6}^6 = d_{C_5}(p') + \frac1{128} \b^{36}.$$ But then from \eqref{eq:dupper}, we have 
\[
d_{C_5}(p') + \frac1{128} \b^{36} \le d_{C_5}(G') \le d_{C_5}(p')+  O(\eps + \a),
\]
a contradiction for small $\eps$ since $\b =  (\eps+\a)^{1/100}$.
\end{proof} 

Without loss of generality say that $|X_1|, \ldots, |X_\ell| \ge \b n$ and $|X_i| < \b n$ for any $i>\ell$. By Claim~\ref{claim:balanced}, there is some value $x$ such that $|X_i| \in [(x - \b) n, (x + \b) n]$ for $1 \le i \le \ell$. Then the number of edges in $G'$ is at most
\begin{align*}
  \binom{n}{2} - \sum \binom{|X_i|}{2} &\le \binom{n}{2} - \ell \binom{(x-\b) n}{2}\\
  & = \frac12 n^2 (1-\ell x^2 + O\rbrac{ \b}).
 \end{align*}
We will now show a lower bound matching the above upper bound. Since for any numbers $a\ge b$ and $\delta>0$, we have $(a+\delta)^2 + (b-\delta)^2 > a^2+b^2$ the following holds. Since $\sum_{i>\ell} |X_i| \le n$, and for $i > \ell$ we have $|X_i| \le \b n$, the maximum possible value of $\sum_{i>\ell} |X_i|^2$ occurs when all the terms are either $0$ or $(\b n)^2$, meaning that the number of positive terms would be at most $\frac{1}{\b }$, so we have 
\[
\sum_{i>\ell} |X_i|^2 \le \frac{1}{\b } \cdot (\b n)^2 = \b n^2
\]
 the number of edges in $G'$ is then at least
\begin{align*}
  \binom{n}{2} - \sum \binom{|X_i|}{2} &\ge \binom{n}{2} - \ell \binom{(x+\b) n}{2} -\frac12 \b n^2\\
    & = \frac12 n^2 (1-\ell x^2 +O(\b) ).
 \end{align*}
But we know $G'$ has edge density $p' = 1 - \frac1k + O(\a)=1-\ell x^2 +O(\b)$ and so we get 
\[
x=\frac{1}{\sqrt{k\ell}} + O(\b)
\]
and in particular $\ell \le k$ since otherwise $|X_1| + \ldots + |X_\ell| \ge (\ell x + O(\b))n > n$. To summarize, at this point we know that the graph must have $\ell \le k$ ``large" parts which each have about $\frac{1}{\sqrt{k\ell}} n$ vertices, and the rest of the parts are ``small" and each have at most $\b n$ vertices. We would like to show that $\ell = k$, so assume for contradiction that $\ell < k$. 
\begin{claim}
$\sum_{i > \ell} |X_i| > \b n$.
\end{claim}
\begin{proof}
Observe that 
\[
\sum_{i > \ell} |X_i| = n-\sum_{i \le \ell} |X_i|= n-\ell\rbrac{\frac{1}{\sqrt{k\ell}} + O(\b)}n = \rbrac{1-\frac{\sqrt{\ell}}{\sqrt{k}}+O(\b)}n > \b n
\]
since $\ell < k$ and we may assume $\b>0$ is arbitrarily small. 
\end{proof}


Now we will use Lemma~\ref{lem:tech}  on $J=G'$ and $X$ being $X_1$ together with several of the small $X_i$s, which will finish the proof.
Recall we have $|X_1|$ of size  $\left( \frac{1}{\sqrt{kl}} + O(\beta) \right)n$. 
We know $|X_i| < \beta n$ for all $i > \ell$ and at the same time $| \cup_{i > \ell} X_i | > \beta n$. 
Hence there exists an integer $z$ such that $\beta n \leq  | \cup_{z \geq i > \ell} X_i | \leq 2\beta n$. 
Let $Y = \cup_{z \geq i > \ell} X_i$. 
In order to apply Lemma~\ref{lem:tech} to $X=X_1 \cup Y$, we need to count the number of copies of $P_4$ in $X$,
the other assumptions of  Lemma~\ref{lem:tech} are clearly satisfied.
Notice that $\nu_{P_4}(X)$ is bounded from below by the number of copies of $P_4$ that alternate vertices in $X_1$ and in $Y$,
which gives
\begin{equation}\label{eq:nuP4}
\nu_{P_4}(X) \geq |X_1|^2 |Y|^2 \geq |X_1|^2 (\beta n)^2 = \frac{\beta^2}{kl} n^4 + O(\beta^3)n^4.
\end{equation}

Denote $|X|$ by $x$.
Notice that
\[
x = |X_1| + |Y| =  \left(\frac{1}{\sqrt{kl}} + O(\beta)\right) n.
\]
Let $e$ be the number of edges in $X$. It can be bounded from above by pretending that $Y$ is a complete graph,
which gives
\[
e \leq |X_1| \cdot |Y| + |Y|^2/2 \leq 	 \frac{2\beta n^2}{\sqrt{kl}} + O(\beta^2)n^2.
\]
This gives 
\[
q = \frac{2e}{x^2} \leq  4\beta \sqrt{kl} + O(\beta^2).
\]
Hence $X$ satisfies of Lemma~\ref{lem:tech}\textit{(iii)} with $\eps' = \frac{\beta^2kl}{2}$, since
\[
\frac{1}{2}x^4 q^3 \leq \frac{32\beta^3}{\sqrt{kl}}n^4 + O(\beta^4)n^4
\]
is significantly smaller than $\nu_{P_4}(X)$ (see \eqref{eq:nuP4}) and 
$\eps' x^4 \leq \frac{\beta^2}{2kl}n^4 + O(\beta^4)n^4$ is about $\frac12 \nu_{P_4}(X)$.
Hence Lemma~\ref{lem:tech} implies
\[
d_{C_5}(G') \geq  d_{C_5}(p') +  \frac{\beta^{12}(kl)^6}{2^7} >  d_{C_5}(p') + \beta^{19}. 
\]
Combining this with \eqref{eq:dupper} gives the final contradiction
\[
d_{C_5}(p') + \beta^{19}  \le d_{C_5}(G') \le d_{C_5}(p')+  O(\eps + \a)
\]
for a small $\eps$ since $\b =  (\eps+\a)^{1/100}$.
\end{proof}

Summarizing, we just showed that $G$ can be transformed into the Tur\'an graph $T_{n}^k$ by adding or deleting at most $o(n^2)$ edges.

\subsection{Second proof}

Here we use some notions related to graphons. An introduction to graphons and further details can be found in the excellent book by Lov\'asz~\cite{Lovasz:lngl}.
In general, a \emph{graphon} is a quadruple $Q=(\Omega,\C B,\mu,W)$, where $(\Omega,\C B,\mu)$ is a standard probability space and $W:\Omega\times\Omega\to[0,1]$ is a symmetric measurable function, see~\cite[Section~13.1]{Lovasz:lngl}. For a graph $F$ on $[k]$, its \emph{induced homomorphism density} in $Q$ is 
 \begin{equation*}
\tind(F,Q)= \int_{\Omega^k} \prod_{ij\in E(F)} W(x_i,x_j) \prod_{ij\in E(\O F)} (1-W(x_i,x_j))\,\dd \mu(x_1)\ldots\dd \mu(x_k). 
\end{equation*}
 Here we identify two graphons $Q$ and $Q'$ if $\tind(F,Q)=\tind(F,Q')$ for every graph $F$, calling them \emph{equivalent}.
 
The relevance of graphons comes from the result of Lov\'asz and Szegedy~\cite{LovaszSzegedy06} that positive homomorphisms $\phi:Hom^+(\C A,\I R)\to\I R$ are in one-to-one correspondence with graphons $Q$ (up to equivalence) so that, for every graph $F$, we have $\phi(F)=p(F,Q)$, where we let $p(F,Q)=\frac{|F|!}{|\aut(F)|}\, \tind (F,Q)$ with
$\aut(F)$ being the group of automorphisms of $F$. Also, let 
 $$
 d_{C_5}(Q)=\frac1{5!}\sum_{F\in\C F_5} c_F^{OPT} p(F,Q).
 $$

We correspond to a graph $G=(V,E)$  the graphon $Q_G=(V,2^V,\mu,A)$ where $\mu$ is the uniform measure and $A:V\times V\to \{0,1\}$ is the adjacency function of~$G$.
Then, for example, $\tind(F,Q_G)$ is the probability that a uniform random map $f:V(F)\to V(G)$ is an \emph{induced homomorhism}, that is, for all $i,j\in V(F)$, $\{i,j\}\in E(F)$ if and only if $\{f(i),f(j)\}\in E(G)$. We say that a sequence of graphons $Q_n$ \emph{converges} to $Q$ 
if, for every graph $F$, we have $\lim_{n\to\infty} \tind(F,Q_n)=\tind(F,Q)$. In particular, if $Q_n=Q_{G_n}$ for some increasing sequence of graphs $G_n$, then this gives the same convergence of graphs that we used.

Since, by Lemma~\ref{lm:NoCocherries}, we will be seeing only the limits of (almost) complete partite graphs, the following more restrictive class $\C P$ of ``complete partite'' graphons will suffice for our purposes. 
Namely, from now on, we fix $\Omega$ to be the set $\{0,1,2,\dots\}$ of non-negative integers with the discrete topology (thus all subsets of $\Omega$ or $\Omega^2$ are measurable) and fix $W(i,j)$ to be 0 if $i=j\ge 1$ and be 1 otherwise (i.e., if $i\not=j$ or if $i=j=0$). Only the measure $\mu$ will vary, and the measures that we consider are as follows. Let
$$
 \C{R}=\left\{\rho\in [0,1]^{\I N}\mid \rho_1\ge \rho_2\ge\dots,\ \sum_{i=1}^\infty \rho_i\le 1\right\}.
 $$ 
 For $\rho=(\rho_1,\rho_2,\dots)\in\C{R}$, define the probability measure $\mu_\rho$ on $(\Omega,2^\Omega)$ by $\mu_\rho(\{i\})=\rho_i$ for $i\ge1$. Thus $\mu_\rho(\{0\})=\rho_0$, where $\rho_0$ is always a shorthand for $1-\sum_{i=1}^\infty \rho_i$ (but is not an entry of the vector $\rho=(\rho_1,\rho_2,\dots)$).
 Also, define
 $$
  P_\rho=(\Omega,2^\Omega,\mu_\rho,W),\qquad\mbox{for $\rho\in\C{R}$}
  $$
  and let $\C P=\{P_\rho\mid \rho\in\C{R}\}$ consist of all graphons that arise this way.

For example, a complete partite graph $G$ gives a graphon $P_G\in\C P$ as follows. Order the parts $V_1,\dots,V_s$ of $G$ non-increasingly by their size, let $\rho_G=(|V_1|/|G|,\dots,|V_s|/|G|,0,0,\dots)$, and take $P_G=(\Omega,2^\Omega,\mu_{\rho_G},W)$. Since all vertices inside a part $V_i$ are twins in $G$, we have that $\tind(F,Q_G)=\tind(F,P_G)$ for every $F\in\C F$. Thus $Q_G$ and $P_G$ are equivalent graphons.

One should think of $P_\rho$ as the limit of complete partite graphs where, for $i\ge 1$, $\rho_i$ is the fraction of vertices in the $i$-th largest part while $\rho_0$ is the total fraction of vertices in parts of relative size~$o(1)$.

\begin{lemma}\label{lm:pointwise} If a sequence of vectors $\rho_{n}\in \C{R}$ converges to $\rho\in [0,1]^{\I N}$ in the product topology (that is, pointwise), then $\rho\in\C{R}$ and
the corresponding graphons $P_{\rho_{n}}$ converge to $P_\rho$.\end{lemma}

\begin{proof} If $\sum_{i=1}^\infty \rho_i>1$, then $\sum_{i=1}^m \rho_i>1$ for some $m$ and thus $\sum_{i=1}^m \rho_{n,i}>1$ for sufficiently large $n$, a contradiction. Thus $\rho\in\C{R}$.

We have to show that the graphons $P_n=(\Omega,2^\Omega,\mu_n,W)$ converge to $P_{\rho}$, where $\mu_n=\mu_{\rho_{n}}$. Take any $F\in\C F$ and $\e>0$. Let $k=|F|$ and fix an integer $m>3{k\choose 2}/\e$. 

For any $Q=(\Omega,2^\Omega,\mu,W)\in\C P$, define $Q'=(\Omega,2^\Omega,\mu',W)\in\C P$, where $\mu'$ is the push-forward of the measure $\mu$ under the map that sends each $i>m$ to $0$ and is the identity otherwise. (In the $\C{R}$-domain, this corresponds to truncating $x\in\C{R}$ to $x'=(x_1,\dots,x_m,0,\dots)\in\C{R}$.) Let us show that  \begin{equation}\label{eq:Q'}
 |\tind(F,Q)-\tind(F,Q')|\le \e/3,\qquad\mbox{for every $Q\in\C P$.}
 \end{equation} 
 This inequality becomes more obvious if we allow general graphons and observe that the graphon $Q'$ is equivalent to $(\Omega,2^\Omega,\mu,W')$, where $W'(i,j)$ is defined to be $0$ if $1\le i=j\le m$ and $1$ otherwise. Thus when we pass from $W$ to $W'$ on the same probability space $(\Omega,2^\Omega,\mu)$, then for every $i\in \Omega$ the measure of $j$ with $W(i,j)\not=W'(i,j)$ is always at most $\frac{1}{m+1}\le\e/3{k\choose 2}$. By Tonelli's theorem, this also upper bounds the $\mu^2$-measure of the set $Z$ of pairs in $\Omega^2$ where $W$ and $W'$ differ. Now, $\tind(F,\cdot)$ is an integral of a $[0,1]$-valued function over $\Omega^k$ and, by the Union Bound, the probability that some pair hits $Z$ is at most ${k\choose 2}\mu^2(Z)\le \e/3$, giving the desired.

Note that $\mu_n'(\{i\})=\mu_n(\{i\})$ converges to $\mu_\rho'(\{i\})=\mu_\rho(\{i\})$ for each $i\in[m]$. It follows that $\mu_n'(\{0\})$ converges to $\mu_\rho'(\{0\})$,
since the support of probability measures $\mu_\rho'$ and any $\mu_n'$ is a subset of $\{0\}\cup[m]$. For such measures $\tind(F,\cdot)$ is a polynomial (and thus continuous) function of the measures of singletons $0,\dots,m$. Thus, for all large $n$, we have that $|\tind(F,P_n')-\tind(F,P_\rho')|\le \e/3$; then it holds by~\eqref{eq:Q'} that  $|\tind(F,P_n)-\tind(F,P_\rho)|\le \e$. Since $\e>0$ and $F$ were arbitrary, $P_n\to P_\rho$ as required.\end{proof} 

\begin{remark} Using some standards facts about graphons, one can prove the converse implication of Lemma~\ref{lm:pointwise} (namely that the graphon convergence $P_{\rho_{n}}\to P_\rho$ implies that $\rho_{n}\to\rho$), see~\cite{LiuPikhurkoSharifzadehStaden} where the space $\C P$ is studied in more detail.
\end{remark}	

Note that the limit of the Tur\'an graphs $T^k_n$ as $n\to\infty$ is $Q_{K_k}$ (or, equivalently, $P_\rho$ for $\rho=(\frac1k,\dots,\frac1k,0,\dots)\in\C{R}$).

\begin{lemma}\label{lm:Unique} For every $k\ge 3$, every sequence of graphs as in~\eqref{eq:Gm} converges to $Q_{K_k}$.
\end{lemma}

\begin{proof} Let us first show that $(G_n)_{n=1}^\infty$ has a subsequence convergent to $Q_{K_k}$. By Lemma~\ref{lm:NoCocherries} and the induced removal lemma, we can make each $G_n$
complete partite, without changing the convergence of any subsequence. Recall that $\rho_{G_n}\in\C{R}$ is the vector encoding the part ratios of $G_n$. Since the product space $[0,1]^{\I N}$ is compact, some subsequence of $\rho_{G_n}\in[0,1]^{\I N}$ converges to some $\rho$. By Lemma~\ref{lm:pointwise}, we have that $\rho\in \C{R}$ and the corresponding subsequence of graphs $G_n$ converges to $Q=P_\rho$. Thus the graphon $Q$ satisfies that
$p(K_2,Q)=p$ and $d_{C_5}(Q)=\lambda$.

The identity in~\eqref{eq:main} can be re-written as an identity valid for every graphon. Since we need to analyse it only for $Q$, let us state a version that uses the (very special) structure of graphons in $\C P$. We need a few definitions first.

For a graph $F\in\C F^1$ on $[k]$ rooted at $1$ and $j\in\Omega$, define the \emph{rooted density} of $F$ in $(Q,j)$ as $\tind(F,(Q,j))=\sum_f \prod_{i=2}^k \rho_{f(i)}$, where $f$ in the sum ranges over all maps $V(F)\to \Omega$ such that $f(1)=j$ and, for all distinct $u,v\in V(F)$, we have that $W(f(u),f(v))=1$ if and only if $\{u,v\}\in E(F)$. Equivalently, this is the limit as $n\to\infty$ of the probability of the following event $\mathcal{E}$. Suppose we choose $k-1$ independent uniform vertices in $G_n$ together with another vertex we call the {\em root} which we make adjacent to everybody else if $j=0$ or $G_n$ has fewer than $j$ parts, and otherwise we put the root in the $j$-th largest part of $G_n$. Then we let $\mathcal{E}$ be the event that these chosen vertices together with the root induce a vertex-labeled homomorphic copy of $F$. For example, if $H\in\C F$ is the unrooted copy of $F$, then 
 \begin{equation}\label{eq:AvOfTind}
 \tind(H,Q)=\sum_{j=0}^\infty \tind(F,(Q,j))\, \rho_j.
 \end{equation}
The version for unlabeled non-roots is $p(F,(Q,j))=\frac{(k-1)!}{|\aut(F)|}\, \tind(F,(Q,j))$, where $\aut(F)$ is the group of root-preserving automorphisms of $F$. We also define a column vector  
$$
Y_j=\big(p(X_1,(Q,j)),\dots,p(X_6,(Q,j))\big)^T\in\I R^6,
$$
 where $X=(X_1,\dots,X_6)^T$ was defined in~\eqref{eq:X}. With this notation, the limit version of~\eqref{eq:main} 
 is
\begin{equation}\label{eq:identity}
5! d_{C_5}(Q)-\sum_{F\in\C F_5} c_F p(F,Q)+\alpha\left(p(K_2,Q)-p\right)=\sum_{j=0}^\infty   Y_j^T MY_j\, \rho_j.
\end{equation}

Recall that each $c_F$ in~\eqref{eq:identity} is at least $5!\lambda$ and that $p(K_2,Q)=p$ for our $Q$ (which is the limit of some $G_n$); thus the left-hand size of~\eqref{eq:identity} is non-positive. Also, recall that $M\succeq0$; thus $x^TMx\ge 0$ for every $x\in\I R^6$ with equality if and only if $Mx=0$. As the $3\times 3$-matrix $A$ in the factorization~\eqref{eq:BAB} is non-singular, the null-space $N$ of $M$ is the same as that of $B$. Calculations (see e.g.\ the Maple code in Appendix~\ref{appendix:B}) show that the 3-dimensional vector space $N$ can be spanned by $\V z_1,\V z_2,\V z_3\in\I R^6$ where
	\begin{equation}\label{eq:Zs}
	\left(\begin{array}{c}\V z_1^T\\ \V z_2^T\\ \V z_3^T
	\end{array}\right)
	= \left(
	\begin{array}{cccccc}
	1 & 0 & 0 & 2 (k-1) & k-1 & k^2-3 k+2 \\
	0 & 1 & 0 & -2 & 0 & 1 \\
	0 & 0 & 1 & k-1 & \frac{k-2}2 & \frac{k^2-3 k+2}2 
	\end{array}
	\right).
	\end{equation}
	 
	 By the previous paragraph, $Y_j\in\I R^6$ belongs to $N$ for every $j\in\Omega$ with $\rho_j>0$. 
	 Since $N$ is a (finite dimensional and thus closed) linear subspace, it also contains the mean $\O{Y}=\sum_{j=0}^\infty Y_j\, \rho_j$. By~\eqref{eq:AvOfTind}, we have that, in particular, $\O{Y}_2=\tind(\O{P_2},Q)$
	and $\O{Y}_3=\frac12\,\tind(\O{P_2},Q)$ are both 0. Since the entries in each $Y_j$ are non-negative,
	we conclude that 
 $Y_j$ has its second and third coordinates zero for every $j\in\Omega$ with $\rho_j>0$.
	The row-reduced matrix in~\eqref{eq:Zs} shows that such $Y_j$ must be colinear to $\V z_1$.
	Since the sum of entries of $Y_j$ is 1, we have $Y_j=\frac1{k^2}\V z_1$. In particular, its first coordinate is $1/k^2$. On the other hand, it is $p(X_1,(Q,j))$ which is the density of $\O{K_3}$ rooted at $j$ in $Q$, that is, it is $\rho_j^2$ if $j\ge 1$ and 0 if $j=0$. Thus $\rho_0=0$ and $\rho_j=1/k$ for every $j$ in the support of $\mu$, so indeed $Q=Q_{K_k}$ is the limit of Tur\'an graphs.

Finally, if we assume  on the contrary to the lemma that  the whole sequence $(G_n)_{n=1}^\infty$ does not converges to $Q_{K_k}$, then by the compactness of the space of all graphons, some subsequence converges to a graphon non-equivalent to $Q_{K_k}$. But then this violates the first claim of the proof.\end{proof}

\begin{proof}[Second Proof of Theorem~\ref{th:stab}.] Lemma~\ref{lm:Unique} and the fact that each graphon is the limit of some sequence of finite graphs imply that the limit version of the $C_5$-minimisation problem has the unique solution $Q_{K_k}$ whose function $W$, moreover, happens to be $\{0,1\}$-valued. These are exactly the assumptions of~\cite[Theorem~15]{Pikhurko10dm} which directly gives the required stability property.
	
	In order to give the reader some idea of what is going on, let us unfold slightly the proof of~\cite[Theorem~15]{Pikhurko10dm} for this particular case. Suppose on the contrary that, for some integer $k\ge 3$ and $\delta>0$, a sequence $G_n$ of graphs as in~\eqref{eq:Gm} violates the stability property. By passing to a subsequence, it converges to some graphon $Q$ with $p(K_2,Q)=p$ and $d_{C_5}(Q)=\lambda$.
	By Lemma~\ref{lm:Unique}, we can assume that $Q=Q_{K_k}$. While the convergence of $G_n$ to some graphon does not identify $G_n$ within edit distance $o(|G_n|^2)$ in general,
	it does if the function $W$ of the graphon assumes only values $0$ and $1$, see~\cite[Lemma~2.9]{LovaszSzegedy10ijm} or~\cite[Theorem~17]{Pikhurko10dm}. 
	Thus, the convergence $G_n\to Q_{K_k}$ implies that $G_n$ is $o(|G_n|^2)$-close to $T_k^{|G_n|}$ in the edit distance, contradicting our assumption.\end{proof}

Another possible derivation of Theorem~\ref{th:stab} from Lemma~\ref{lm:Unique}
is to use the known properties of the so-called \emph{cut-distance} via the argument in~\cite[Page 146]{PikRaz2017} where the description of all extremal graphons for the triangle-minimisation problem was used to describe all almost extremal graphs.

\section{Remarks on the case $p \neq 1-\frac{1}{k}$}\label{sec:general_case}

Our general upper bound construction is as follows. Suppose that $p$ is a constant satisfying $1-\frac1k < p < 1-\frac1{k+1}$. Partition the vertices into $k-1$ sets $X_1, \ldots, X_{k-1}$ of size $xn$ and one more  set $Y$ of size $yn$. Each $X_i$ is an independent set. For $1 \le i \neq j \le k-1$ we have that $X_i$ is complete to $X_j$. $Y$ is also complete to each $X_i$. Finally, $G[Y]$ is any graph such that for some parameter $0 < \rho < \frac{1}{2}$ we have
\begin{enumerate}[(i)]
\item \label{cond1}$G[Y]$ has asymptotically $\frac12 y^2n^2 \rho$ edges, $\frac12 y^3 n^3 \rho^2$ paths of length 2 (that means on 3 vertices), and $\frac12 y^4n^4 \rho^3$ paths of length 3;
\item \label{cond2}$G[Y]$ has $o(n^5)$ copies of $C_5$.
\end{enumerate}
(See the end of this subsection for discussion on which graphs are suitable for $G[Y]$.)
We assume that 
\[
(k-1)x+y = 1
\]
so we have a total of $n$ vertices. The edge density in this construction is
\[
\frac{\binom{k-1}2 \rbrac{xn}^2 + (k-1)(xn)(yn) + (\frac12+o(1)) y^2n^2 \rho}{\binom{n}{2}}, 
\]
which tends to 
\[
g(x, y, \rho) = (k-1)_2 x^2 + 2(k-1)xy + \rho y^2
\]
as $n \rightarrow \infty$. So we also assume that the parameters $x, y, \rho$ satisfy $g(x, y, \rho) = p$. 

Now we consider the ratio $f(x, y, \rho) = \lim_{n\to\infty} \frac{\nu_G(C_5)}{n^5}$.
We claim that
\begin{align*}
f(x, y, \rho)&=  \sbrac{\frac1{10} (k-1)_5 + \frac12 (k-1)_4 + \frac12 (k-1)_3} x^5\\ 
&\,+ \sbrac{ \frac12(k-1)_4 + \frac32 (k-1)_3 +  \frac12(k-1)_2} x^4 y\\
&\,+ \sbrac{ \rbrac{\frac12 + \frac12 \rho}(k-1)_3 + \rbrac{1 + \frac12 \rho}(k-1)_2    } x^3 y^2\\
&\,+ \sbrac{\rbrac{\frac12 \rho+ \frac12 \rho^2}(k-1)_2  + \frac12 \rho(k-1) } x^2 y^3 \\
&\,+ \frac12 \rho^3 (k-1) x y^4.
\end{align*}
Note that we have grouped the terms of $f(x, y, \rho)$ according to powers of $x$ and $y$, and then according to falling factorials of $(k-1)$. To understand our formula, it helps to think of the powers of $x,y$ as specifying how many vertices come from sets of size $xn, yn$, and the falling factorial $(k-1)$ as specifying how many distinct sets of size $xn$ are involved. For example, the first term $\frac1{10} (k-1)_5 \; x^5$ is there because there are $\frac1{10} (k-1)_5 (xn)^5$ many copies of $C_5$ having vertices $v_1, \ldots, v_5$ all in different parts of size $xn$. Now let us justify a more complicated term like say the second term in the third line, $ \rbrac{1+\frac12 \rho}(k-1)_2\; x^3y^2 $. This term counts the copies of $C_5$ that have vertices $v_1, \ldots, v_5$ such that $v_1$ and $v_2$ come from $Y$, $v_3$ and $v_4$ are in the same set of size $xn$, and $v_5$ is in some other set of size $xn$ (and $v_1, \ldots, v_5$ may be in any order on the cycle). The case where $v_1$ and $v_2$ are consecutive in the cycle contributes $\frac12 (k-1)_2 \rho (yn)^2(xn)^3$, and the other case contributes $(k-1)_2 (yn)^2(xn)^3$.

Now for a given integer $k\ge 2$ and a real number $1-\frac1k < p < 1-\frac1{k+1}$ we define an optimization problem (P):
\begin{center}
\begin{tabular}{l l l}
& \texttt{Minimize}    & $f(x, y, \rho)$ \\[8pt]
     & \texttt{subject to:} & $(k-1)x+y = 1$,\\[5pt]
     & & $g(x, y, \rho) = p$,\\[5pt]
     & & $x,y\ge 0$.\\[5pt]
\end{tabular}
\end{center}
Let us denote its solution by $f_{\min}(p) = f(x_0, y_0, \rho_0)$. Clearly, $d_{C_5}(p) \le f_{\min}(p)$. For some certain values of $k$ and $p$ 
we verified that $120\cdot f_{\min}(p)$ numerically matches 
%
the lower bound on $d_{C_5}(p)$ given by the flag algebras.
In particular, when we calculated with unlabeled flags of order $\ell$, we were getting numerically matching 
bounds for $p \leq 1 - \frac{1}{\ell-2}$ and we observed a gap in the bounds for $p > 1 - \frac{1}{\ell-2}$ different from Tur\'an densities.
Since computer calculations can be performed with current computers in a reasonable time only for $\ell \leq 8$, a simple straightforward use of computer is unlikely to provide a numerical match of $d_{C_5}(p)$ and $f_{\min}(p)$ for all $p$.
Unfortunately, we were unable to convert the numerical match to a formal proof.
The main problem is that (P) has no closed solution. For example, for $k=2$ and $\frac{1}{2}< p < \frac{2}{3}$ we can plug into the objective function $y=1-x$ and $\rho = (p - x^2 - 2xy)/y^2$ obtaining
\[
f(2, x, 1-x, (p - x^2 - 2xy)/y^2) = {\frac { x ( 2x^{2}-2x+p)  ( 3x^{4}-5x^{3} +(1+4p)x^{2}+ (1-4p)x+p^2) }{2 \left(x-1
 \right) ^{2}}}.
\]
Now it is not difficult to show that there exists a local minimum for some $\frac{1}{3} < x < \frac{1}{2}$. Unfortunately, it looks like this minimum can be only found numerically.
There might be a different parametrization of the problem that would make it possible to solve $(P)$ and formally show  a match with flag algebra calculations for some range of $p$.
On Figure~\ref{fig:opt} we present the shape of $f_{\min}(p)$. We conjecture that $d_{C_5}(p) = f_{\min}(p)$ for any~$p$.

\begin{figure}
\centering
\begin{subfigure}[b]{0.3\textwidth}
                \centering
                \includegraphics[width=\textwidth]{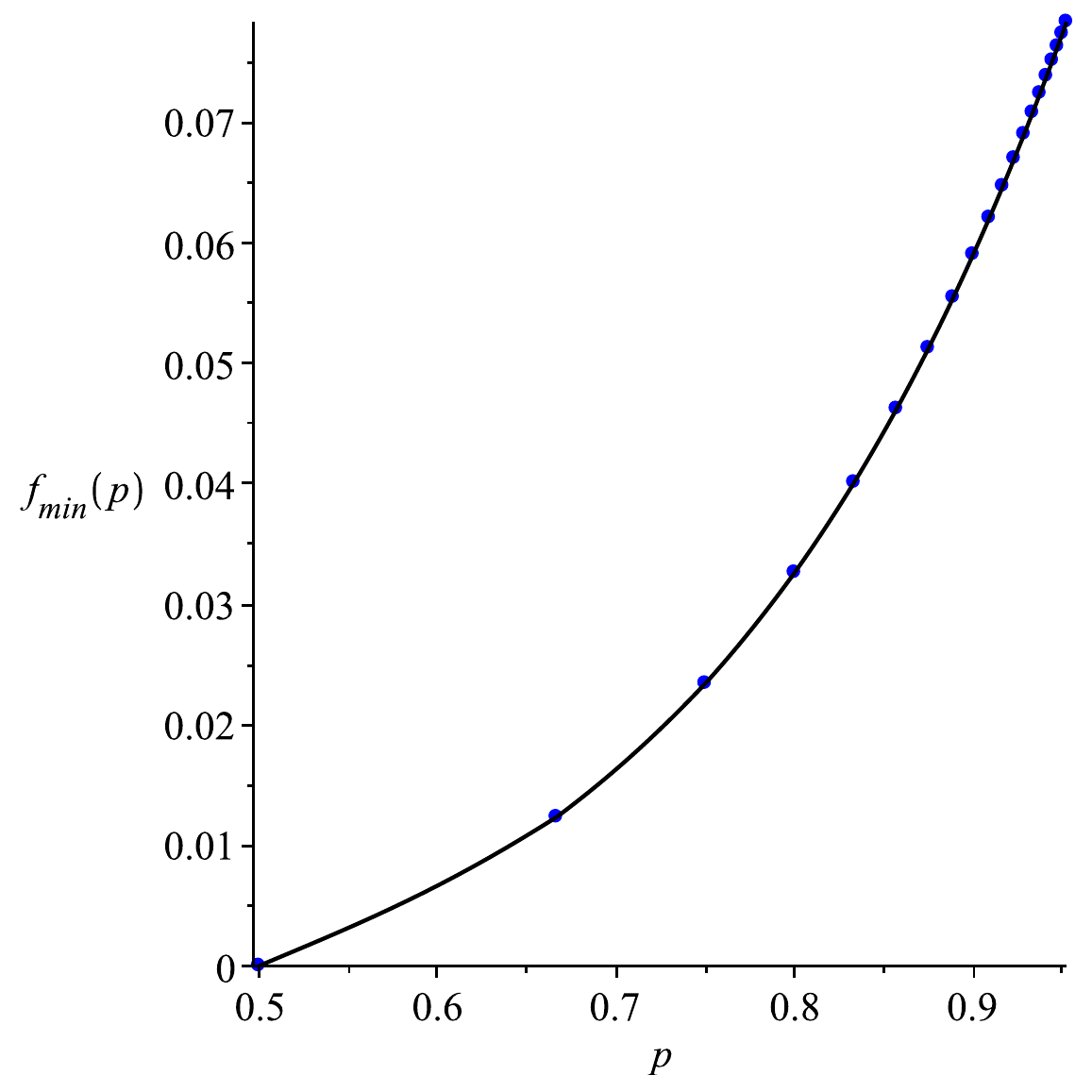}
                \caption{}
                \label{subfig:1}
\end{subfigure}
\quad
\begin{subfigure}[b]{0.3\textwidth}
                \centering
                \includegraphics[width=\textwidth]{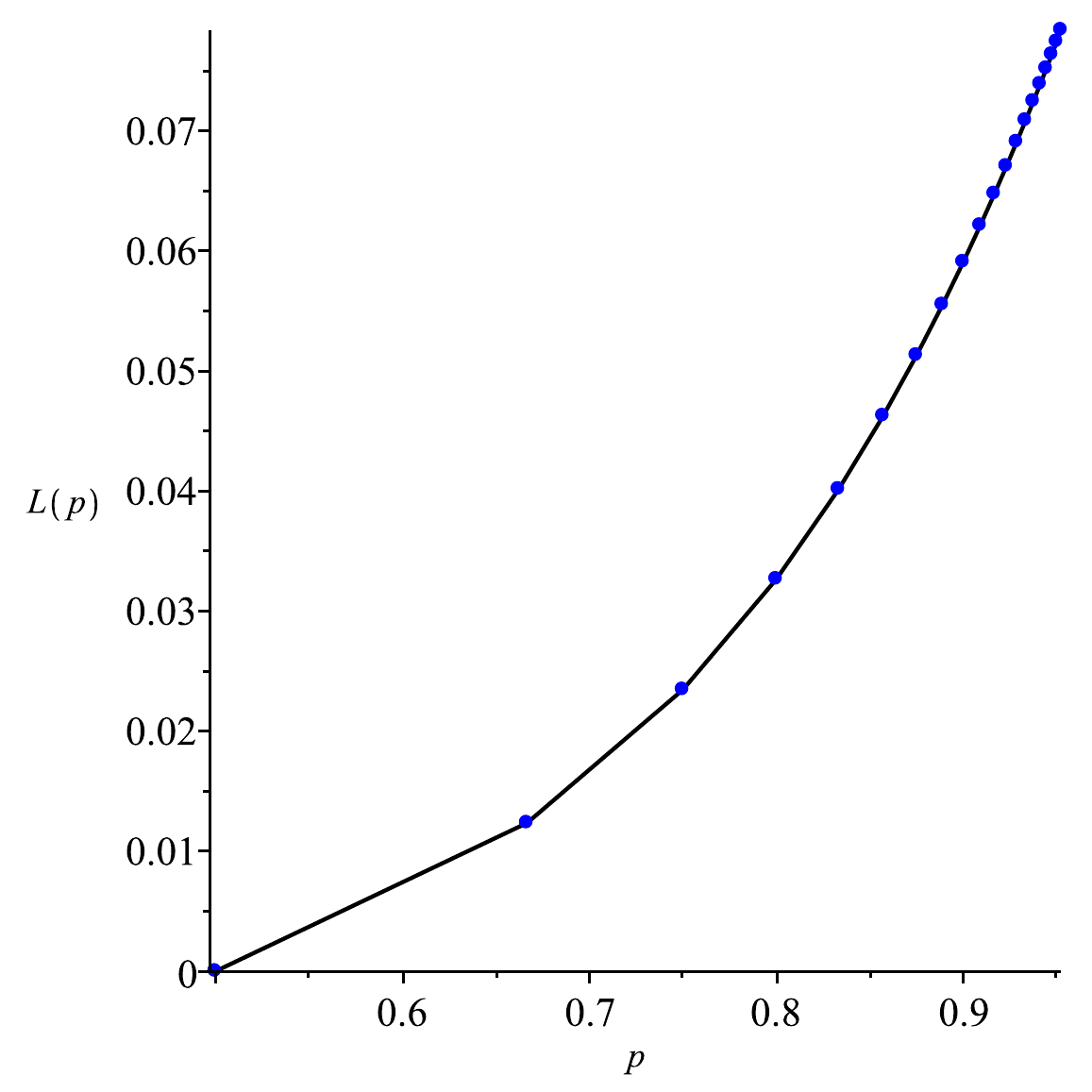}
                \caption{}
                \label{subfig:2}
\end{subfigure}%
\quad
\begin{subfigure}[b]{0.3\textwidth}
                \centering
                \includegraphics[width=\textwidth]{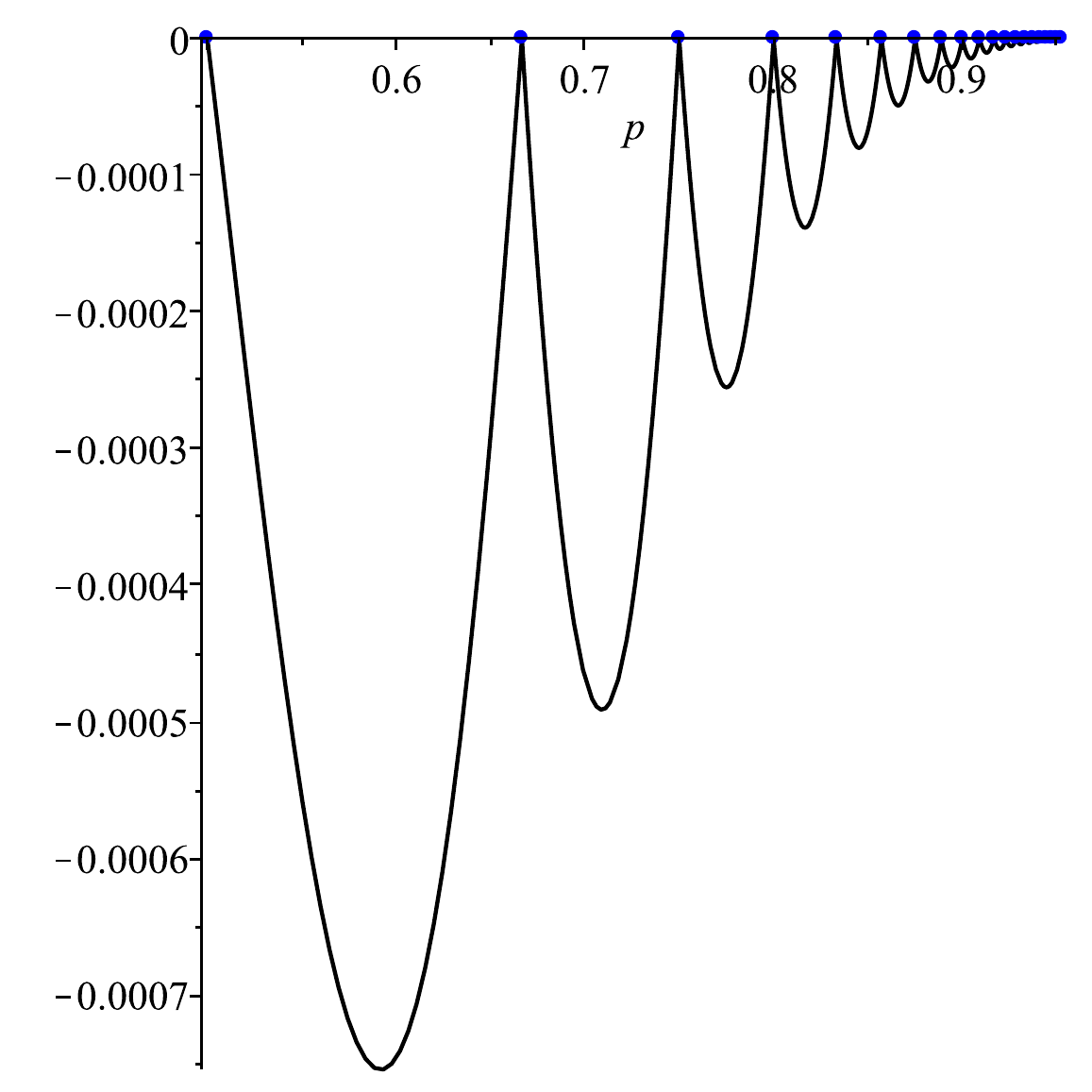}
                \caption{}
                \label{subfig:3}
\end{subfigure}%
\caption{(\subref{subfig:1}) A graph of $f_{\min}(p)$ based on numerical calculations. Blue points correspond to the Tur\'an densities (i.e. $p=1-1/k$). (\subref{subfig:2}) Secant lines between Tur\'an densities. (\subref{subfig:3}) A graph of $f_{\min}(p) - L(p)$.}
\label{fig:opt}
\end{figure}

We now address what graphs are suitable for $G[Y]$, i.e. what graphs satisfy \eqref{cond1} and \eqref{cond2}. Note first that some such choice of $G[Y]$ exists, for example it can be a random bipartite graph with two parts of size $\frac12 yn$ and edge probability $2 \rho$. Now we claim that $G[Y]$ satisfies \eqref{cond1} if and only if $G[Y]$ is {\it almost $yn\rho$-regular,} or more formally, all but $o(n)$ vertices in $G[Y]$ have degree $(1+o(1))yn\rho$. Indeed, if $G[Y]$ is almost $yn\rho$-regular then it is easy to verify the edge and path counts in \eqref{cond1}. Conversely, suppose \eqref{cond1} holds, and let the random variable $Z$ represent the degree of a random vertex in $G[Y]$. Then we have $\E[Z] = (1+o(1))yn\rho$ and since $\sum_{v \in Y} \binom{deg(v)}{2}$ is the number of paths of length 2 we can calculate
\begin{align*}
\E[Z^2] = & \frac1{yn} \sum_{v \in V(Y)} deg(v)^2= \frac1{yn}\cdot 2(1+o(1))\frac12 y^3n^3 \rho^2= (1+o(1))y^2 n^2 \rho^2 = (1+o(1))\E[Z]^2
\end{align*}
so $Z$ is concentrated by Chebyshev's inequality (see, e.g., Lemma~20.3 in~\cite{FriKar2016}). In other words, $G[Y]$ is almost $yn \rho$-regular. 

We believe that we have described all almost optimal graphs. Specifically, we believe that any graph with edge density $p$ and $C_5$-density $d_{C_5}(p)+o(1)$ can be transformed by adding or deleting at most $o(n^2)$ edges into a graph with a vertex partition $X_1, \ldots, X_{k-1}, Y$ where $|X_i|=xn, |Y|=yn$, all $X_i$ are independent, all $X_i$ and $Y$ are complete to each other, and $G[Y]$ is $yn\rho$-regular where $x, y, \rho$ are a solution to the optimization problem (P).
\providecommand{\bysame}{\leavevmode\hbox to3em{\hrulefill}\thinspace}
\providecommand{\MR}{\relax\ifhmode\unskip\space\fi MR }
\providecommand{\MRhref}[2]{%
  \href{http://www.ams.org/mathscinet-getitem?mr=#1}{#2}

}

\providecommand{\bysame}{\leavevmode\hbox to3em{\hrulefill}\thinspace}
\providecommand{\MR}{\relax\ifhmode\unskip\space\fi MR }
\providecommand{\MRhref}[2]{%
  \href{http://www.ams.org/mathscinet-getitem?mr=#1}{#2}
}
\providecommand{\href}[2]{#2}

\appendix

\newpage 

\section{Appendix}\label{appendix:A}


\ \vspace{-0.5cm}

\renewcommand{\arraystretch}{1.3}
{\setlength{\tabcolsep}{1pt}

\begin{table}[H]
    \centering
    \begin{adjustbox}{angle=90}
 \scalebox{0.4}{

}
\end{adjustbox}
\caption{All entries corresponding to $p(K_2,F)$ are multiplied by 10 and all entries corresponding to $\llbracket X_i \times X_j \rrbracket_1$ are multiplied by 30.}
\label{table:pF}
\end{table}

\newpage

\section{Appendix}\label{appendix:B}

This Maple code computes $c_F$ coefficients. Matrices $A$, $B$ and $M$ are defined in Subsection~\ref{subsec:nice_p}. $X$ is a matrix of size $21 \times 34$ and it is defined in Appendix~\ref{appendix:A} (rows correspond to $\llbracket X_i \times X_j \rrbracket_1$). Vectors \verb+cFOPT, pF, cFM+ and \verb+cF+ (each of size 34) correspond to $\ c_F^{OPT}, p(K_2,F), c_F^M$ and $c_F$, respectively. Constant $a$ corresponds to $\alpha$.

\bigskip

\begin{verbnobox}[\footnotesize]
restart:
with(LinearAlgebra):
A := Matrix([[32*k^2-96*k+96, 0, 4*k^2-16*k], 
[0, 10*k^4-30*k^3-8*k^2+96*k-96, -10*k^4+35*k^3-4*k^2-80*k+96], 
[4*k^2-16*k, -10*k^4+35*k^3-4*k^2-80*k+96, 10*k^4-40*k^3+24*k^2+64*k-96]]):
B := Matrix([[k-1, 1, k-2, 0, k-3, -1], 
[0, 2, k-2, 0, 2*k-4, -2], 
[0, 0, k-1, -1, 2*k-2, -2]]):
M:= (3/(2*k^4))*Matrix(Multiply(Transpose(B), Multiply(A, B))):
X:=(1/30)*Matrix([[30,12,4,0,0,0,4,2,0,0,0,0,2,0,0,0,0,0,0,0,0,0,0,0,0,0,0,0,0,0,0,0,0,0],
[0,3,4,3,0,6,0,1,2,0,0,0,0,1,0,0,0,0,0,0,0,0,0,0,0,0,0,0,0,0,0,0,0,0],
[0,6,4,3,0,0,8,2,0,6,2,0,0,0,2,0,1,0,0,0,0,0,0,0,0,0,0,0,0,0,0,0,0,0],
[0,0,2,6,12,0,0,2,2,0,3,4,0,0,0,2,0,1,0,0,0,0,0,0,0,0,0,0,0,0,0,0,0,0],
[0,0,1,0,0,0,0,2,0,0,1,0,4,0,1,0,2,0,3,0,0,0,0,0,0,0,0,0,0,0,0,0,0,0],
[0,0,0,0,0,3,0,0,2,0,0,2,0,2,0,1,0,2,0,3,0,0,0,0,0,0,0,0,0,0,0,0,0,0],
[0,0,0,0,0,0,2,2,2,0,0,0,4,4,0,0,0,0,0,0,6,0,0,0,0,0,0,0,0,0,0,0,0,0],
[0,0,0,0,0,0,4,2,1,4,2,0,0,0,2,2,0,0,0,0,0,6,2,1,0,0,0,0,0,0,0,0,0,0],
[0,0,0,0,0,0,0,0,0,2,2,2,0,0,2,2,2,2,0,0,0,0,1,2,3,0,0,0,0,0,0,0,0,0],
[0,0,0,0,0,0,0,0,0,1,0,0,0,0,2,0,1,0,0,0,0,0,1,0,0,0,5,2,0,1,0,0,0,0],
[0,0,0,0,0,0,0,0,0,0,0,0,0,0,0,0,0,0,0,0,0,3,2,1,0,4,0,1,2,0,1,0,0,0],
[0,0,4,0,0,12,0,4,4,4,0,0,0,0,6,2,0,0,0,0,0,12,4,0,0,0,10,2,0,0,0,0,0,0],
[0,0,0,0,0,0,0,2,2,0,2,4,8,4,2,0,4,2,0,0,0,0,4,2,0,8,0,2,2,0,0,0,0,0],
[0,0,0,3,0,0,0,0,2,0,2,0,0,2,0,2,1,0,0,0,0,0,2,2,0,0,0,2,0,2,0,0,0,0],
[0,0,0,0,0,0,0,0,1,0,0,0,0,4,0,2,0,2,0,0,12,0,0,3,6,0,0,0,2,0,2,0,0,0],
[0,0,0,0,0,0,0,0,0,0,0,0,0,0,2,2,4,4,12,12,0,0,0,0,0,0,10,6,4,4,4,0,0,0],
[0,0,0,0,0,0,0,0,0,0,1,0,0,0,0,2,2,1,6,0,0,0,0,2,0,0,0,2,2,4,0,4,0,0],
[0,0,0,0,0,0,0,0,0,0,0,0,0,0,0,0,0,0,0,0,0,0,1,2,3,0,0,2,2,6,4,8,6,0],
[0,0,0,0,6,0,0,0,0,0,0,4,0,0,0,0,0,2,0,0,0,0,0,0,0,4,0,0,2,0,0,2,0,0],
[0,0,0,0,0,0,0,0,0,0,0,1,0,0,0,0,0,2,0,6,0,0,0,0,3,0,0,0,1,0,4,0,3,0],
[0,0,0,0,0,0,0,0,0,0,0,0,0,0,0,0,0,0,0,0,0,0,0,0,0,2,0,0,2,0,4,4,12,30]]):
cFM := Vector(34):
k_ind := 0:
printlevel := 2:
for i to 6 do 
    for j from i to 6 do 
        k_ind := k_ind+1; 
        if i = j then cFM := cFM+M(i, j)*Transpose(Row(X, k_ind)); 
                 else cFM := cFM+2*M(i, j)*Transpose(Row(X,k_ind)); 
        end if; 
    end do; 
end do:
cFOPT := Vector([0,0,0,0,0,0,0,0,0,0,0,0,0,0,0,0,0,0,0,0,0,0,0,0,0,0,1,1,1,2,2,4,6,12]):
pF := (1/10)*Vector([0,1,2,3,4,3,2,3,4,3,4,5,4,5,4,5,5,6,6,7,6,4,5,6,7,6,5,6,7,7,8,8,9,10]):
a := (1/(k^3))*(60*k^3 - 240*k^2 + 360*k - 192):
cF := Vector(34):
for i to 34 do 
    cF(i) := cFOPT(i)-a*pF(i)-cFM(i)+(k-1)*a/k
end do:
for i to 34 do 
    printf("5*k^4*cF(
end do:
kernel := NullSpace(B):
kernelMatrix := Matrix(convert(kernel, list)):
ReducedRowEchelonForm(Transpose(kernelMatrix))
\end{verbnobox}

\end{document}